\documentclass [11pt,reqno]{amsart}
\usepackage {amsmath,amssymb,verbatim,geometry}
\newif\ifpdf
\ifpdf
\usepackage[pdftex]{hyperref}
\else
\fi

\newcommand{\C}{\mathbf{C}}

\newcommand{\Q}{\mathbf{Q}}
\newcommand{\R}{\mathbf{R}}
\newcommand{\Z}{\mathbf{Z}}
\newcommand{\N}{\mathbf{N}}

\newcommand{\fa}{\mathfrak{a}}
\newcommand{\fb}{\mathfrak{b}}
\newcommand{\fm}{\mathfrak{m}}
\newcommand{\fD}{\mathfrak{D}}
\newcommand{\fM}{\mathfrak{M}}
\newcommand{\fL}{\mathfrak{L}}
\newcommand{\fP}{\mathfrak{P}}
\newcommand{\fX}{\mathfrak{X}}

\newcommand{\fY}{\mathfrak{Y}}

\newcommand{\cD}{\mathcal{D}}
\newcommand{\cE}{\mathcal{E}}
\newcommand{\cF}{\mathcal{F}}

\newcommand{\cL}{\mathcal{L}}
\newcommand{\cM}{\mathcal{M}}

\newcommand{\cO}{\mathcal{O}}

\newcommand{\cU}{\mathcal{U}}

\newcommand{\cZ}{\mathcal{Z}}
\newcommand{\cX}{\mathcal{X}}
\newcommand{\cY}{\mathcal{Y}}

\newcommand{\hcO}{\widehat{\mathcal{O}}}

\newcommand{\Xan}{X}
\newcommand{\Xqm}{X^{\mathrm{qm}}}
\newcommand{\Xdiv}{X^{\mathrm{div}}}

\newcommand{\Uan}{U^{\mathrm{an}}}

\renewcommand{\a}{\alpha}
\renewcommand{\b}{\beta}
\renewcommand{\d}{\delta}
\newcommand{\e}{\varepsilon}
\newcommand{\f}{\varphi}

\newcommand{\g}{\gamma}

\newcommand{\om}{\omega}
\newcommand{\p}{\psi}

\newcommand{\alme}{{\rm a.e.\ }}

\newcommand{\ie}{{\rm i.e.\ }}

\newcommand{\an}{\mathrm{an}}

\DeclareMathOperator{\one}{\mathbf{1}}

\DeclareMathOperator{\Capa}{Cap}
\DeclareMathOperator{\Spec}{Spec}

\DeclareMathOperator{\MA}{MA}
\DeclareMathOperator{\MAC}{M}

\DeclareMathOperator{\supp}{Supp}
\DeclareMathOperator{\vol}{Vol}
\DeclareMathOperator{\Pic}{Pic}

\DeclareMathOperator{\id}{id}

\DeclareMathOperator{\ord}{ord}

\DeclareMathOperator{\Proj}{Proj}
\DeclareMathOperator{\NS}{NS}
\DeclareMathOperator{\PSH}{PSH}

\DeclareMathOperator{\Div}{Div}

\DeclareMathOperator{\spec}{Spec}

\DeclareMathOperator{\emb}{emb}

\newcommand{\retr}{p}
\newcommand{\cent}{c}

\newcommand{\ft}{\varphi^{\langle t\rangle}}
\newcommand{\ftt}{\varphi^{\langle t/2\rangle}}
\newcommand{\fs}{\varphi^{\langle s\rangle}}
\newcommand{\ftj}{\varphi^{\langle t\rangle}_j}
\newcommand{\ps}{\psi^{\langle s\rangle}}
\newcommand{\pt}{\psi^{\langle t\rangle}}

\newcommand{\D}{\Delta}

\newcommand{\cro}[1]{[\![#1]\!]}
\newcommand{\lau}[1]{(\!(#1)\!)}

\numberwithin{equation}{section}       % Number formulas within sections

\newtheorem{prop} {Proposition} [section]
\newtheorem{thm}[prop] {Theorem} 
\newtheorem{defi}[prop] {Definition}
\newtheorem{lem}[prop] {Lemma}
\newtheorem{cor}[prop]{Corollary}
\newtheorem{prop-def}[prop]{Proposition-Definition}

\newtheorem*{thmA}{Theorem A} 
\newtheorem*{thmAp}{Theorem A'}

\newtheorem{rmk}[prop]{Remark}
\theoremstyle{remark}
\newtheorem*{assertAp}{Assertion A(p)}

\newtheorem*{ackn}{Acknowledgment} 

\title{Solution to a non-Archimedean Monge-Amp\`ere equation}
\date{\today}

\author{S{\'e}bastien Boucksom
  \and
  Charles Favre
  \and
  Mattias Jonsson}

\address{CNRS--Universit{\'e} Pierre et Marie Curie\\
  Institut de Math{\'e}matiques\\
  F-75251 Paris Cedex 05\\
  France}
\email{boucksom@math.jussieu.fr}

\address{CNRS--CMLS\\
  \'Ecole Polytechnique\\
  F-91128 Palaiseau Cedex\\
  France}
\email{favre@math.polytechnique.fr}

\address{Dept of Mathematics\\
  University of Michigan\\
  Ann Arbor, MI 48109-1043\\
  USA}
\email{mattiasj@umich.edu}

\begin{document}

\begin{abstract}
Let $X$ be a smooth projective Berkovich space over a complete discrete valuation field $K$ of residue characteristic zero, and assume that $X$ is defined over a function field admitting $K$ as a completion. Let further $\mu$ be a positive measure on $X$ and $L$ be an ample line bundle such that the mass of $\mu$ is equal to the degree of $L$. Then we show the existence a continuous semipositive metric whose associated measure is equal to $\mu$ 
%$\|\cdot\|$ on $L$ such that $c_1(L,\|\cdot\|)^{\dim X}=\mu$ 
in the sense of Zhang and Chambert-Loir. This we do under a technical assumption on the support of $\mu$, which is, for instance, fulfilled if the support is a finite set of divisorial points. Our method draws on analogues of the variational approach developed to solve complex Monge-Amp\`ere equations on compact K\"ahler manifolds by Berman, Guedj, Zeriahi and the first named author, and of Ko{\l}odziej's $C^0$-estimates. It relies in a crucial way on the compactness properties of singular semipositive metrics, as defined and studied in a companion article. 
\end{abstract}

\maketitle

\setcounter{tocdepth}{1}
\tableofcontents

\newpage
\section{Introduction}\label{S100}
The goal of this paper is to construct continuous solutions to a
non-Archimedean analogue of certain complex Monge-Amp\`ere equations
on projective manifolds, which arose in complex geometry as more
degenerate versions of the by-now classical equations considered by
Aubin, Calabi and Yau. More specifically, our main result can be
understood as an analogue of a fundamental result by S.~Ko{\l}odziej~\cite{kolodziej1}. 

\medskip
Let us briefly recall the complex statement that we have in
mind. Let $L$ be an ample line bundle on a smooth complex
projective variety $X$ of dimension $n$. Let $\mu$ be a positive measure
on $X$, of mass equal to $c_1(L)^n$. It was shown in \cite{kolodziej1}
that under a mild regularity assumption on $\mu$ (which is for instance satisfied as soon as $\mu$ has $L^p$-density with respect to Lebesgue measure for some $p>1$), there exists a
continuous metric $\|\cdot\|$ on $L$, unique up to a
multiplicative factor, whose curvature form $c_1(L,\|\cdot\|)$ is a
closed positive $(1,1)$-current satisfying $c_1(L,\|\cdot\|)^n=\mu$ in
the sense of pluripotential theory \cite{BT1}. 
This result relied on the work of Aubin, Calabi and Yau, 
which culminated in the celebrated article~\cite{Yau}, 
where it was shown that the solution metric is smooth when 
$\mu$ is a smooth positive volume form on $X$.

\medskip
We next turn to the non-Archimedean analogue, referring to~\S\ref{S101} for more details. Let $K$ be a complete discrete valuation field whose residue field $k$ has characteristic zero, so that $K\simeq k\lau{t}$. Let $X$ be a smooth projective variety over $K$, and write $n=\dim X$. Thanks to the non-Archimedean GAGA principle, it is reasonable to also denote by $X$ the corresponding $K$-analytic space in the sense of Berkovich, whose underlying topological space is compact Hausdorff. A \emph{model} of $X$ is a normal scheme $\cX$ that is flat and projective over $S:=\Spec k\cro{t}$, and whose generic fiber can be identified with $X$. 

Consider a  ample line bundle $L$ on $X$. A \emph{model metric} on $L$ is a metric defined by a extension $\cL\in\Pic(\cX)_\Q$ of $L$ to some model $\cX$. Such a metric is called \emph{semipositive} if $\cL$ is nef, \ie has non-negative degree on all proper curves of the special fiber of $\cX$. S.-W.~Zhang introduced in~\cite{zhang} the more flexible notion of \emph{semipositive continuous} metric as the uniform limit of semipositive model metrics.\footnote{We refer to Table~\ref{table} in~\S\ref{S302} below for a comparison of our terminology with existing terminology.}
In this context, A. Chambert-Loir~\cite{Ch1} defined the \emph{Monge-Amp\`ere measure} $c_1(L,\|\cdot\|)^n$ of a semipositive continuous metric $\|\cdot\|$ on $L$. It is a positive Radon measure on $X$, of mass
$\deg L$. 

V.~Berkovich constructed in \cite{Ber2} the \emph{skeleton} associated to a polystable model of $X$. Since we are assuming $K$ to have residue characteristic zero, it is easier to rely on resolution of singularities  and  instead consider  \emph{SNC models}, \ie models whose special fiber has simple normal crossing support (but is not necessarily reduced, as opposed to a semistable model). To each SNC model $\cX$ is associated a \emph{dual complex} 
$\D_\cX$ that encodes the combinatorics of the intersections 
of the components of the special fiber, and which 
embeds in the Berkovich space $X$ just as skeletons do. Any finite set of divisorial points
is contained in the dual complex of some SNC model; in particular $\bigcup_\cX\D_\cX$ is dense in $X$.

We can now state our main result.
We say  that $X$ is \emph{algebraizable}  if there exists a (one-variable) function field $F$ admitting $K$ as a completion and a smooth projective $F$-scheme $Y$ such that $X=Y_K$.
\begin{thmA}
Let $K$ be a complete discrete valuation field of residue characteristic zero. Let $X$ be a smooth projective $K$-variety that is \emph{algebraizable}. Let $L\in\Pic(X)$ be an ample line bundle and $\mu$ be a positive Radon measure on $X$ of mass $c_1(L)^n$. If we further assume that $\mu$ is supported on the dual complex of some SNC model of $X$, then there exists a continuous, semipositive metric $\|\cdot\|$ on $L$ such that
\begin{equation}\label{eq:nama}
  c_1\left(L,\|\cdot\|\right)^{\dim X} = \mu~. 
\end{equation}
This metric is furthermore unique up to a multiplicative constant.  
\end{thmA}
Even though the result is most likely true without this assumption, the algebraizability condition plays an essential role in our proof, as we shall explain below. Note that the line bundle $L$ is not assumed to be defined over a function field. 

\medskip
The uniqueness part in Theorem A follows from a result of X.~Yuan and S.-W.~Zhang \cite{yuanzhang} asserting more generally that a continuous semipositive metric $\|\cdot\|$ is uniquely determined up to a constant by its Monge-Amp\`ere measure. Their proof is inspired by the one given by B\l{o}cki~\cite{Blocki} in the complex setting. 

Our approach does not give any information on the regularity of the metric
besides continuity. It would be interesting to further investigate
this issue, for instance when $\mu$ is supported on finitely many divisorial points. We refer to \S\ref{sec:examples} for a discussion of this problem in the case of toric varieties, based on the recent work~\cite{BPS}.

\medskip
Versions of Theorem~A are already known in a few cases. For curves (and in fact over any
complete non-Archimedean, non-trivially valued field), it can easily be deduced from results of A.~Thuillier~\cite{thuillierthesis}, who developed a theory of singular semipositive metrics on analytic curves that is completely analogous to the complex case. Solving~\eqref{eq:nama} for curves boils down to
a system of linear equations and relies on the negativity of the intersection form
of the special fiber of a suitable model, see~\S\ref{sec:examples}. Alternatively, one can exploit the structure of the Berkovich space as
a metrized graph as in~\cite{BRbook,valtree}.

In higher dimensions, Y. Liu~\cite{liu} treated the related case when $X$ is a
totally degenerate abelian variety over $\C_p$, and
$\mu$ is a (smooth) measure supported on the dual complex of the canonical formal model of $X$, as constructed by Mumford. By exploiting the fact that this dual complex is a compact (real) torus, one can translate the equation $c_1(L,\|\cdot\|)^n = \mu$ into a (real) Monge-Amp\`ere equation on this real torus, and apply Yau's result to its complexification to obtain the metric. 

\medskip
A statement very close to Theorem~A also appears in an unpublished set of notes by M. Kontsevich and Y. Tschinkel~\cite{KoTs} dating from 2001, where the authors propose a detailed strategy of proof in the case $\mu$ is a Dirac mass at a divisorial point. Several ingredients in their approach also appear in our paper (see Remark \ref{rmk:KT} below). 

\medskip
We are now going to present  an outline of our proof of Theorem A, which consists in mimicking as far as possible the variational approach to complex
Monge-Amp\`ere equations of \cite{BBGZ} and the $C^0$-estimates of \cite{kolodziej1}. To that end we will rephrase Theorem A in a more analytic language. Let us thus recall the notion of quasi-plurisubharmonic function that we developed in \cite{siminag} and its main properties.
 
As a variant of \cite{BGS} we first define the space of \emph{closed $(1,1)$-forms} on $X$ as the direct limit 
$$
\cZ^{1,1}(X):=\varinjlim_\cX N^1(\cX/S),
$$
where $\cX$ ranges over all models of $X$ and the space of \emph{numerical classes} $N^1(\cX/S)$ is defined as $\Pic(\cX)_\R$ modulo numerical equivalence on the special fiber. Each closed $(1,1)$-form $\theta\in\cZ^{1,1}(X)$ defines a class $\{\theta\}\in N^1(X)$, which we refer to as its \emph{de Rham class}. We say that $\theta\in\cZ^{1,1}(X)$ is \emph{semipositive} if it is determined by a nef numerical class on some model. Each model metric $\|\cdot\|$ on a line bundle $L$ over $X$ defines a closed $(1,1)$-form $c_1(L,\|\cdot\|)$ that we call the \emph{curvature form} of the metric. The de Rham class of $c_1(L,\|\cdot\|)$ is just $c_1(L)\in N^1(X)$, and the model metric $\|\cdot\|$ is semipositive (in the sense of Zhang) iff its curvature is. Each model metric on the trivial line bundle is of the form $e^{-\f}$ for some $\f\in C^0(X)$, which is then by definition a \emph{model function}. Following complex notation, we write $dd^c\f$ for the curvature form of this metric, so that $c_1(L,\|\cdot\| e^{-\f})=c_1(L,\|\cdot\|)+dd^c\f$. 

Now let $\om\in\cZ^{1,1}(X)$ be a reference closed semipositive $(1,1)$-form on $X$, such that $\{\om\}\in N^1(X)$ is furthermore ample. This situation arises for instance when $\om$ is the curvature form of a semipositive model metric on an ample line bundle $L$. As was shown in \cite{siminag}, one may then define a class $\PSH(X,\om)$ of $\om$-psh functions with the following properties: 
\begin{itemize}
\item Each $\f\in\PSH(X,\om)$ is an upper semicontinuous function $X\to[-\infty,+\infty[$ whose restriction to the faces of any dual complex is continuous and convex. 
\item The set $\PSH(X,\om)$ is convex and stable under max. 
\item A model function $\f\in\cD(X)$ is $\om$-psh iff $\om+dd^c\f\in\cZ^{1,1}(X)$ is semipositive.
\end{itemize}
The two main results of \cite{siminag} further state that
\begin{itemize}
\item $\PSH(X,\om)/\R$ is compact with respect to the topology of uniform convergence on dual complexes.
\item Every $\f\in\PSH(X,\om)$ is the decreasing limit of a family of $\om$-psh model functions.
\end{itemize}
It follows from the latter property and Dini's lemma that every \emph{continuous} $\om$-psh function is a \emph{uniform} limit over $X$ of $\om$-psh model functions. This shows in particular that our definition of continuous semipositive metrics is compatible with Zhang's. Chambert-Loir's definition of the Monge-Amp\`ere measure of a continuous semipositive metric immediately extends to our setting and enables us to associate to any $n$-tuple of \emph{continuous} $\om$-psh functions $\f_1,...,\f_n\in C^0(X)\cap\PSH(X,\om)$ a (mixed) Monge-Amp\`ere measure
$$
(\om+dd^c\f_1)\wedge...\wedge(\om+dd^c\f_n),
$$
a positive Radon measure on $X$ of mass $\{\om\}^n$, which depends continuously on $(\f_1,...,\f_n)$ with respect to the topology of uniform convergence on $X$. As in the complex case, it is however not possible to define such mixed Monge-Amp\`ere measures in a reasonable way for arbitrary $\om$-psh functions, as soon as $n\ge 2$. 

The following result is a slight generalization of Theorem~A phrased in the present language.
\begin{thmAp}
  Let $X$ be an algebraizable smooth projective $K$-variety as in Theorem A. Let $\om\in\cZ^{1,1}(X)$ be a closed semipositive $(1,1)$-form such that $\{\om\}\in N^1(X)$ is ample and let $\mu$ be a positive Radon measure on $X$ of mass $\{\om\}^n$. If $\mu$ is supported in a dual complex then there exists a continuous $\om$-psh function $\f$ such that
\begin{equation}\label{equ:ma}
(\om+dd^c\f)^n=\mu.
\end{equation}
The function $\f$ is furthermore unique up to an additive constant.
\end{thmAp}
This formulation is designed to emphasize the analogy with the complex case. However, it is important to keep in mind that the non-Archimedean Monge-Amp{\`e}re operator is not a differential operator but rather defined in terms of intersection 
theory. 

\smallskip
Let us now set up the variational approach we use to solve our non-Archimedean Monge-Amp\`ere equation, following \cite{BBGZ}. A key feature of Monge-Amp\`ere equations is that they may be written as Euler-Lagrange equations. This fact goes back at least to Alexandrov~\cite{Alexandrov} in the more classical case of real Monge-Amp\`ere equations, while the relevant functional in the complex case has been well-known in K\"ahler geometry since the works of Aubin, Calabi and Yau. We introduce in our setting the \emph{energy functional} 
\begin{equation}\label{eq:energy}
E_\om(\f) := \frac1{n+1}\sum_{j=0}^{n} \int \f \, (\om + dd^c \f)^j \wedge \om^{n-j},
\end{equation}
defined for the moment for $\f\in C^0(X)\cap\PSH(X,\om)$. An easy computation shows that
\begin{equation}\label{eq:deriv}
\frac{d}{dt}\bigg|_{t=0_+} E_\om((1-t)\f + t\psi) = \int(\psi-\f) \, (\om+dd^c\f)^n
\end{equation}
for any two $\f,\psi\in C^0(X)\cap\PSH(X,\om)$, so that (\ref{equ:ma}) is indeed the Euler-Lagrange equation of the functional
$$
F_\mu(\f):=E_\om(\f)-\int\f\,d\mu.
$$
Observe that the compatibility condition $\mu(X)=\{\om\}^n$ guarantees that $F_\mu$ is translation-invariant, \ie $F_\mu(\f+c)=F_\mu(\f)$ for all $c\in\R$. As in the complex case, one shows that the functional $E_\om$ is concave on $C^0(X)\cap\PSH(X,\om)$, so that any solution $\f$ to (\ref{equ:ma}) is necessarily a maximizer of $F_\mu$. The variational method conversely amounts to proving the existence of a maximizer of $F_\mu$ and showing that it satisfies (\ref{equ:ma}). But the lack of compactness of the space $C^0(X)\cap\PSH(X,\om)$ where $F_\mu$ is defined so far makes it hard to construct a maximizer,  while it is at any rate non-obvious that such a maximizer should satisfy the Euler-Lagrange equation, since it might belong to the boundary of $C^0(X)\cap\PSH(X,\om)$. In order to circumvent these difficulties we are going to argue along the following three steps. 
\begin{itemize}
\item[Step 1:] Enlarge the space where the variational problem is being considered, in order to gain compactness and construct a maximizer $\f_0$ there.
\item[Step 2:] Show that the maximizer is in a natural way a "generalized solution" of the non-Archimedean Monge-Amp\`ere equation (\ref{equ:ma}). 
\item[Step 3:] Show the regularity (\ie continuity) of this generalized solution using capacity estimates.  
\end{itemize}
The general strategy for Steps 1 and 2 follows \cite{BBGZ}, whereas Step 3 follows \cite{kolodziej1}. 

\smallskip
The condition that $\mu$ is supported on a dual complex makes Step 1 relatively easy in our case, granted the compactness property of $\PSH(X,\om)/\R$ proved in \cite{siminag}. Indeed, the support condition guarantees that the linear part $\f\mapsto\int\f\,d\mu$ of $F_\mu$ is finite valued and continuous on the whole of $\PSH(X,\om)$. Because of that, several complications that occurred in \cite{BBGZ} to handle general measures disappear, since it is enough to extend $E_\om$ to a usc functional $E_\om:\PSH(X,\om)\to[-\infty,+\infty[$, which is done by setting
\begin{equation*}
  E_\om(\f):=\inf\left\{ E_\om(\psi)\mid \psi\ge\f,\ \psi\in C^0(X)\cap\PSH(X,\om)\right\}. 
\end{equation*}

\smallskip 
Step 2 requires much more work and constitutes the main body of the article, in particular because virtually none of the more classical results in pluripotential theory on which \cite{BBGZ} was able to rely were available so far in our non-Archimedean context. The only obvious information we have on the maximizer $\f_0$ of $F_\mu$ is that it lies in the set
$$
\cE^1(X,\om):=\left\{\f\in\PSH(X,\om),\,E_\om(\f)>-\infty\right\}
$$ 
of \emph{$\om$-psh functions with finite energy}. In the complex case, $\cE^1(X,\om)$ was introduced in \cite{Ceg2,GZ2} as a higher dimensional and non-linear generalization of the classical Dirichlet space from potential theory. The goal of Step 2 is to show that the Monge-Amp\`ere operator can be naturally extended to $\cE^1(X,\om)$, and that $\f_0$ satisfies 
$$
(\om+dd^c\f_0)^n=\mu
$$
in this generalized sense. 

In order to do so, we first extend the Monge-Amp\`ere operator from continuous to bounded $\om$-psh functions, following the fundamental work of Bedford and Taylor~\cite{BT1,BT2}. As in the complex case, this  mild generalization is in fact crucial in order to develop a reasonable capacity theory, and also because the natural bounded approximants $\max\{\f,-m\}$, $m\in\N$, of a given $\om$-psh function $\f$ are not continuous in general. It is however substantially more involved than the continuous case, since uniform convergence has to be replaced with monotone convergence. The fact that any (bounded) $\om$-psh function can be written as a decreasing limit of a family of $\om$-psh model functions, proved in \cite{siminag}, plays a key role at this stage. 

Of crucial importance is the following \emph{locality property} of the Monge-Amp{\`e}re operator: if $\f,\p$ are bounded $\om$-psh functions, then the restrictions of the measures $(\om+dd^c\max\{\f,\p\})^n$ and $(\om+dd^c\f)^n$ to the Borel set $\{\f>\p\}$ coincide. Note that, even when $\f,\p$ are model functions, this fact is not clear from the definition in terms of intersection numbers. 

Next, we further extend the Monge-Amp\`ere operator from bounded $\om$-psh functions to functions with finite energy. The key observation, which goes back to~\cite{BT2}, is the monotonicity of the sequence of measures
$$
\one_{\{\f>-m\}}\left(\om+dd^c\max\{\f,-m\}\right)^n\,(m\in\N)
$$
a direct consequence of the locality property. This allows us to define $(\om+dd^c\f)^n$ as the increasing limit of this sequence of measures, which is shown to be well-behaved for $\f\in\cE^1(X,\om)$. More generally, mixed Monge-Amp\`ere measures are shown to be well-defined for functions in $\cE^1(X,\om)$, and (\ref{eq:energy}), (\ref{eq:deriv}) are still valid in this generality. 

As was already pointed out, these facts are however \emph{a priori} not enough to show that the maximizer $\f_0$ of $F_\mu$ satisfies $(\om+dd^c\f_0)^n=\mu$, because small perturbations of $\f_0$ cease to be $\om$-psh in general. In order to handle a similar difficulty in the setting of real Monge-Amp\`ere equations, Alexandrov  devised in \cite{Alexandrov} an envelope argument, an analogue of which was subsequently found in the complex case in~\cite{BBGZ}. Following the same lead, we introduce the \emph{$\om$-psh envelope} $P_\om(f)$ of a given continuous function $f$ on $X$ by setting for each $x\in X$
\begin{equation*}
  P_\om(f)(x):=\sup\left\{\f(x)\mid\f\in\PSH(X,\om),\,\f\le f\right\}. 
\end{equation*}
It follows from \cite{siminag} that $P_\om(f)$ is the largest $\om$-psh function dominated by $f$ on $X$. The key point is then the following \emph{differentiability} property, whose complex analogue was established in \cite{BB}:
\begin{equation}\label{eq:diff-env}
 \frac{d}{dt}\bigg|_{t=0} E_\om\circ P_\om\left( f+ t g\right) = \int_{X} g\,(\om+dd^c P_\om(f))^n
\end{equation}
for any two $f,g\in C^0(X)$, which may more vividly be written as the chain rule-like formula $(E_\om\circ P_\om)'=E_\om'\circ P_\om$. Granted~\eqref{eq:diff-env}, a  fairly direct argument based on the monotonicity of $E_\om$ implies $(\om+dd^c\f_0)^n = \mu$ as desired.

The proof of~\eqref{eq:diff-env} can be reduced by elementary arguments to the differentiability of
$t\mapsto \int P_\om(f+ t g) \, (\om+dd^c P_\om(f))^n$, which in turn ultimately
follows from the  following \emph{orthogonality property}: 
\begin{equation}\label{eq:orthogon}
  \int_X (f - P_\om(f)) \, (\om+dd^c P_\om(f))^n =0. 
\end{equation}
Since $f\ge P_\om(f)$, this relation means that $(\om+dd^c P_\om(f))^n$ is supported on the contact locus $\{f=P_\om(f)\}$, a well-known fact in the
complex case where the proof argues by balayage, using Bedford and
Taylor's solution to the Dirichlet problem for the homogeneous complex
Monge-Amp\`ere equation on the ball. Such an approach seems far beyond reach in the non-Archimedean case. We proceed instead by translating (\ref{eq:orthogon}) into an intersection theoretic statement on a model of $X$, where it boils down to the orthogonality of relative asymptotic Zariski decompositions for a line bundle that is ample on the generic fiber. It is precisely at this point that we use the assumption that $X$ is algebraizable. Indeed, this allows us to choose the model where we work to be algebraic, and therefore compactifiable into a projective variety over the residue field $k$. As explained in Appendix A, we can then reduce to the absolute case of big line bundles on projective varieties treated in~\cite{BDPP}.

\smallskip
Finally, Step 3 is handled by adapting in a fairly direct manner
the capacity estimates of Ko\l odziej~\cite{kolodziej1,kolodziej2} 
to prove that $\f_0$ is actually continuous. 
The proof relies on the locality property in $\cE^1(X,\om)$.
This shows the existence part of Theorem~A'.
Uniqueness is proved  following \cite{Blocki}, as in~\cite{yuanzhang}.

\bigskip
Our result is not optimal, and we next discuss three important assumptions 
that we use in Theorems~A and~A'.

First, the condition that the measure $\mu$ be supported on a dual complex
is probably unnecessarily strong. Relying on ideas of
Cegrell~\cite{Ceg2}, Guedj and Zeriahi \cite{GZ2} have defined in the
case of compact K\"ahler manifolds a class $\cE (X,\om)$ of $\om$-psh
functions where the Monge-Amp\`ere operator is well-defined and such that the measures $(\om+dd^c\f)^n$, $\f\in\cE (X,\om)$ are exactly the positive measures
$\mu$ on $X$ giving zero mass to pluripolar\footnote{A subset set
  $A\subset X$ is pluripolar if there exists an $\om$-psh 
function $\f$ such that $A\subset\{\f=-\infty\}$.} sets. The function $\f$ is here again uniquely determined up to an additive constant by its Monge-Amp\`ere measure, as was later shown by Dinew \cite{Dinew}. We expect the corresponding results to be true in our setting, too. The proof would probably require an even more 
systematic development of pluripotential theory in a non-Archimedean 
setting, something that is certainly of interest.

\smallskip
Second, as explained above, the proof of the orthogonality property
(\ref{eq:orthogon}) relies in a crucial way on the algebraizability
assumption for $X$. It would be interesting to drop this condition, which we expect to be an unnecessary restriction. 

\smallskip
Finally, our variational approach uses the compactness of the space 
$\PSH(X,\om)/\R$, which was obtained in~\cite{siminag}. 
The proof of this fact relied heavily on the existence of SNC models, 
which are so far only available in residue characteristic zero. 
It seems to be a challenging task to extend our methods and 
results to local fields and more general complete non-Archimedean
fields. See~\cite{valtree,hiro} for related work in the case of a
trivially valued field.

\bigskip
Let us end this introduction by indicating the structure of the paper.

In~\S\ref{S101} we give the necessary background on Berkovich spaces, 
metrized line bundles, $\om$-psh functions and wedge-products of closed $(1,1)$-forms. We also recall some facts from measure theory.

The next three sections,~\S\S\ref{S102}-\ref{S104}, develop some 
of the basic Bedford-Taylor theory in our non-Archimedean setting.
The definition of the Monge-Amp{\`e}re operator on bounded functions 
and the continuity along decreasing families is carried out
in~\S\ref{S102}.
In~\S\ref{S103} we introduce a Monge-Amp{\`e}re capacity
used to measures the size of subsets of $X$.
We obtain the important result that any $\om$-psh function
is quasicontinuous, i.e.\ continuous outside a set of arbitrarily small capacity. 
We also strengthen the regularization theorem of~\cite{siminag} and
prove that any $\om$-psh function is a decreasing limit of a (countable)
sequence of $\om$-psh model functions. 
Finally, in~\S\ref{S104} we prove the locality property.
The results in~\S\S\ref{S102}--\ref{S104} and even some of the proofs
parallel those in complex analysis (especially the ones on compact
K\"ahler manifolds, see~\cite{GZ1}). However, the non-Archimedean 
results ultimately originate in basic properties of the intersection 
form on models whereas the basic results in the complex case
concern differential operators.

The energy of an $\om$-psh function is introduced in~\S\ref{S105}.
Following~\cite{Ceg2,GZ2} we extend the Monge-Amp{\`e}re operator
to the class $\cE^1(X,\om)$ of $\om$-psh functions with finite
energy and prove that the locality property continues to hold.

In~\S\ref{S106} we introduce $\om$-psh envelopes and prove the related differentiability theorem. This is a key result that leads to the proof of Theorem~A' given in~\S\ref{S110}. It uses the locality property and is based on an orthogonality statement whose proof is given in Appendix~A.
Here the exposition is modeled on~\cite{BB,BBGZ}.
 
We next explain in~\S\ref{S110} how
to get Theorem~A from Theorem~A'. Finally,~\S\ref{sec:examples} discusses the case of curves and toric varieties.

\begin{ackn}
This work has been strongly influenced by the work of M.~Kontsevich
and Y.~Tschinkel. The 2001 colloquium talk of Kontsevich at the Institut de
Math\'ematiques de Jussieu served as a guiding source for
us. We are also grateful to him for showing to us the unpublished
preprint~\cite{KoTs}. We further thank A.~Thuillier for several interesting discussions, and J.-L.~Colliot-Th\'el\`ene for his help with Lemma \ref{lem:NS}. 

Our work was carried out at several institutions including the IHES,
the \'Ecole Polytechnique, and the University of Michigan. 
We gratefully acknowledge their support. 
The second author was partially supported by the ANR-grant BERKO.\@
The third author was partially supported by the CNRS and the NSF.
\end{ackn}
% 
%
%%%%%%%%%%%%%%%%%%%%%%%%%%%%%%%%%%%%%%%%%%%%%%%%%%%%%%%%%%%%%%%%%%%
%
%

\section{Background}\label{S101}
For this section we refer to our companion paper~\cite{siminag} 
for details and further references.
%
%%%%%%%%%%%%%%%%%%%%%%%%%%%%%%%%%%%%%%%%%%%%%%%%%%%%%%%%%%%%%%%%%%%
%
\subsection{Berkovich space and models}\label{S303}
Let $R$ be a complete discrete valuation ring with fraction field $K$ 
and residue field $k$. We shall assume that $k$ has characteristic zero.
We let $t\in R$ be a uniformizing parameter and normalize the corresponding 
absolute value on $K$ by $\log| t|^{-1}=1$. 
Note that $R\simeq k\cro{t}$ and $K\simeq k\lau{t}$, 
see for instance~\cite{Serre}. Write $S:=\spec R$. 

\medskip
Let $X$ be a smooth projective $K$-variety, \ie an integral 
(but not necessarily geometrically integral) smooth projective $K$-scheme. 
A \emph{model} of $X$ is a normal, flat and projective $S$-scheme $\cX$
with $X$ as its generic fiber. 
We denote by $\cX_0$ its special fiber, and by $\Div_0(\cX)$ the group of \emph{vertical Cartier divisors}, \ie those supported in $\cX_0$. We write $\Div_0(\cX)_\R$ accordingly. 

Let $\cM_X$ be the set of all isomorphism classes of models of $X$. 
Given $\cX',\cX$ in $\cM_X$ we write $\cX'\ge\cX$ if there exists a 
morphism $\cX'\to\cX$ obtained by blowing up an ideal sheaf 
co-supported on the special fiber of $\cX$. This turns $\cM_X$ into a directed set.

Given a model $\cX$, let $(E_i)_{i\in I}$ be the set of irreducible
components of the special fiber. 
For each subset $J\subset I$ set $E_J:=\bigcap_{j\in J} E_j$. 
A regular model $\cX$ is an \emph{SNC model}
if  the special fiber has simple normal crossing support
and $E_J$ is irreducible (or empty) for each $J\subset I$.

\medskip
As a topological space, the Berkovich space $X^\an$ attached to the given 
smooth projective $K$-variety $X$ is compact and can be described 
as follows (cf.~\cite[Theorem 3.4.1]{Ber}). 
Choose a finite cover of $X$ by affine open subsets of the form 
$U=\spec A$ where $A$ is a $K$-algebra of finite type. 
The Berkovich space $U^\an$ is defined as the set of all multiplicative 
seminorms $|\cdot|:A\to\R_+$ extending the given absolute value of $K$, 
endowed with the topology of pointwise convergence. 
The space $X^\an$ is obtained by gluing the open sets $\Uan$.

There is a natural equivalence of categories between projective $K$-analytic spaces and projective $K$-schemes, see~\cite[\S3.4]{Ber}. In the sequel we shall therefore always identify a projective $K$-scheme with its associated Berkovich space and write $X^\an = X$.

Let $\cX$ be a model of $X$. To each irreducible component
$E$ of the special fiber is associated a divisorial 
valuation $\ord_E$ of the function field of $X$. 
After rescaling and exponentiating,
this gives rise to an element $x_E\in X$ called a 
\emph{divisorial point}. The set $\Xdiv$ of divisorial points
is dense in $X$.

When $\cX$ is an SNC model, we can refine this construction.
Write the special fiber as $\cX_0=\sum_{i\in I}b_iE_i$.
The \emph{dual complex} $\D_\cX$ of $\cX$ is the simplicial
complex whose vertices correspond to the irreducible
components $E_i$ and whose simplices correspond to
nonempty intersections $E_J$.
We can equip $\D_\cX$ with an (integral) affine structure
and embed it in the Berkovich space $X$ as follows.

Consider a subset $J\subset I$ with $E_J\ne\emptyset$ and pick 
$w=(w_j)_{j\in J}$ with $w_j\ge 0$ and $\sum_{j\in J}b_jw_j=1$. Let $\xi_J$ be the 
generic point of $E_J$ and pick a system 
$(z_j)_{j\in J}$ of regular parameters for $\cO_{\cX,\xi_J}$
with $z_j$ defining $E_j$. 
By Cohen's structure theorem,
$\widehat{\cO}_{\cX,\xi_J}\simeq\kappa(\xi_J)[[z_j, j\in J]]$.
Let $v_{J,w}$ be the restriction to $\cO_{\cX,\xi}$ of the 
monomial valuation on this power series ring,
taking value $w_j$ on $z_j$, \ie 
$v_{J,w}\left(\sum_{\a\in\N^J}c_\a z^\a\right)=\min\left\{\sum_{j\in J}w_j\a_j\mid c_\a\ne0\right\}$. Then $e^{-v_{J,w}}\in X$. This defines an embedding $\emb_\cX:\D_\cX\to X$, and the parameters $w$ equip $\D_\cX$ with an affine structure.

There is also a \emph{retraction} $\retr_\cX:X\to\Delta_\cX$,
defined as follows. Any point $x\in X$ admits a \emph{center}
on $\cX$. This is the unique point $\xi=\cent_{\cX}(x)\in\cX_0$ such that 
$|\f|_x\le 1$ for $\f\in\cO_{\cX,\xi}$ and $|\f|_x<1$ for $\f\in\fm_{\cX,\xi}$.
Let $J\subset I$ be the maximal subset such that $\xi\in E_J$.
Then $\retr_\cX(x)\in\D_\cX$ corresponds to the monomial valuation with
weight $-\log|z_j|_x$, $j\in J$.

We have $\retr_\cX=\id$ on $\D_\cX$. If $\cY$ dominates $\cX$,
then $\D_\cX\subset\D_\cY$ and $\retr_\cX\circ\retr_\cY=\retr_\cX$.
The retractions induce a homeomorphism of $X$
onto the inverse limit $\varprojlim\D_\cX$.

\smallskip

In order to keep notation light, we shall identify $\D_\cX$ with its image in $X$ under $\emb_\cX$. Note that this convention differs from the one adopted in~\cite{siminag}.
A point in $X$ lying in some dual complex $\D_\cX$ is called quasi-monomial, and 
the set of such points is denoted by $\Xqm$.

%
%%%%%%%%%%%%%%%%%%%%%%%%%%%%%%%%%%%%%%%%%%%%%%%%%%%%%%%%%%%%%%%%%%%
%
\subsection{Model functions}
Let $\cX$ be a model of $X$. 
A vertical fractional ideal sheaf $\fa$ is a finitely generated
$\cO_\cX$-submodule of the function field of $\cX$ such that $\fa|_X=\cO_X$.
Then $\fa$ defines a continuous 
function $\log|\fa|\in C^0(\Xan)$ by setting
\begin{equation*}
  \log|\fa|(x):=\max\left\{\log|f|_x\mid f\in\fa_{\cent_\cX(x)}\right\}. 
\end{equation*}
Note that each vertical Cartier divisor $D\in\Div_0(\cX)$ defines a vertical fractional ideal sheaf $\cO_\cX(D)$, hence a continuous function 
$f_D:=\log|\cO_\cX(D)|$. 
Note that $f_{\cX_0}$ is the constant function $1$ since $\log|t|^{-1}=1$. 
The map $D\mapsto f_D$ extends by linearity to $\Div_0(\cX)_\R\to C^0(\Xan)$. 
\begin{defi}
  A function $f$ on $\Xan$ is  a \emph{model} function
  if there exists a model $\cX$ and a $\Q$-divisor $D \in \Div_0(\cX)_\Q$
  such that $f = f_D$. We then call $\cX$ a \emph{determination} 
  of $f$.
  We let $\cD(\Xan)=\cD(\Xan)_\Q$ be the space of model functions on $\Xan$.
\end{defi}

\begin{prop}\cite[Proposition~2.2]{siminag} The $\Q$-vector space $\cD(X)$ of model functions is stable under max. If $f$ is a model function and $\cX$ is a determination then $f$ is affine on each face of $\D_\cX$.
\end{prop}

%
%%%%%%%%%%%%%%%%%%%%%%%%%%%%%%%%%%%%%%%%%%%%%%%%%%%%%%%%%%%%%%%%%%%
%
\subsection{Forms and de Rham classes}\label{sec:forms}
Let $\cX$ be a model of $X$.  
The space $N^1(\cX/S)$ of (relative, codimension $1$) numerical equivalence classes on $\cX$
is defined as the quotient of $\Pic(\cX)_\R$ by the subspace spanned by 
numerically trivial line bundles, \ie those $\cL\in\Pic(\cX)_\R$ 
such that $\cL\cdot C=0$ for all projective curves contained in a 
fiber of $\cX\to S$. It is in fact enough to consider vertical curves, \ie those contained in the special fiber $\cX_0$. A class $\theta \in N^1(\cX/S)$ is \emph{nef} if 
$\theta \cdot C \ge 0$ for all such curves $C$.
\begin{defi} 
  The space of \emph{closed $(1,1)$-forms} on $\Xan$ is defined as the direct limit 
  \begin{equation*}  
    \cZ^{1,1}(\Xan):=\varinjlim_{\cX\in\cM_X} N^1(\cX/S).
  \end{equation*}
\end{defi}
We say that a closed $(1,1)$-form $\theta\in\cZ^{1,1}(\Xan)$ is 
\emph{determined} on a given model $\cX$ if it is the
image of an element $\theta_\cX\in N^1(\cX/S)$. 
By definition, two classes $\theta\in N^1(\cX/S)$ 
and $\theta'\in N^1(\cX'/S)$ define the same element in
$\cZ^{1,1}(\Xan)$ iff they 
pull back to the same class on a model dominating both $\cX$ and $\cX'$. 
\begin{defi}
  A closed $(1,1)$-form $\theta\in\cZ^{1,1}(X)$ is 
  \emph{semipositive} if $\theta_\cX\in N^1(\cX/S)$ is nef for some
  (or, equivalently, any) determination $\cX$ of $\theta$.
\end{defi}
The natural map $N^1(\cX/S) \to N^1(X)$ gives rise to a 
map $\cZ^{1,1}(\Xan)\to N^1(X)$ which in fact is surjective.
We refer to $\{\theta\}$ as the \emph{de Rham class} of the 
closed $(1,1)$-form $\theta$.
When $\theta$ is semipositive, the de Rham class $\{\theta\}\in N^1(X)$
is nef on $X$. In what follows, we shall mainly work with forms
having ample de Rham class.

Any model function $f\in\cD(X)$ induces a form
$dd^cf\in\cZ^{1,1}(X)$ as follows: for any determination $\cX$ of $f$,
$dd^cf$ is the class of the divisor $\sum_{i\in I}b_if(x_i)E_i$, 
where $\cX_0=\sum_i b_i E_i$ and $x_i\in X$ is the divisorial point associated
to $E_i$.
% 
%%%%%%%%%%%%%%%%%%%%%%%%%%%%%%%%%%%%%%%%%%%%%%%%%%%%%%%%%%%%%%%%%%%
%
\subsection{$\theta$-psh functions}\label{S301}
Fix a form $\theta\in\cZ^{1,1}(\Xan)$
with ample de Rham class $\{\theta\}\in N^1(X)$. 
\begin{defi}\label{defi:psh} 
  A \emph{$\theta$-psh function} 
  $\f:\Xan\to[-\infty,+\infty[$ 
  is an usc function such that for each SNC model $\cX$ of $X$
  on which $\theta$ is determined we have
  \begin{enumerate}
  \item[(i)] 
    $\f\le\f\circ \retr_{\cX}$ on $\Xan$;
  \item[(ii)] 
    the restriction of $\f$ to the dual complex $\D_\cX$ is a 
    uniform limit 
    of restrictions of model functions $\p$ 
    such that $\theta + dd^c \p$ is a semipositive form. 
  \end{enumerate}
  We write $\PSH(\Xan,\theta)$ for the set of $\theta$-psh functions on $X$.
\end{defi}
It is a nontrivial fact that if $\f$ is a $\theta$-psh model function
then the form $\theta+dd^c\f$ is in fact semipositive, 
see~\cite[Theorem~5.11]{siminag}. 
In particular, the zero function is $\theta$-psh iff $\theta$ is semipositive.
In this case,  $\max\{\f,-t\}$ is $\theta$-psh when $\f$ is $\theta$-psh and 
$t\in\R$.

\begin{prop}\cite[Proposition~5.10]{siminag}.\label{Pgenerators}
  The space of model functions $\cD(\Xan)$ is spanned by $\theta$-psh
  model functions. 
\end{prop}

\begin{prop}\cite[Proposition~7.4]{siminag}.\label{Ponly}
The set $\PSH(\Xan,\theta)$ is convex. If $\f,\p$ are $\theta$-psh and $c\in\R$, then the functions $\max\{\f,\p\}$ and $\f+c$ are also $\theta$-psh. 
\end{prop}

\begin{prop}\cite[Proposition~7.5]{siminag}.\label{P302}
  Any $\f\in\PSH(X,\theta)$ is continuous on the dual complex of any SNC model $\cX$, and convex on each of its faces.
\end{prop}
In fact, the continuity statement above can be made \emph{uniform} in $\f$:
\begin{thm}~\cite[Corollary~7.7]{siminag}\label{Pconvex}
  For any SNC model $\cX$, the restrictions of all $\theta$-psh functions
  to the dual complex $\D_\cX$ form an equicontinuous family. 
\end{thm}

We endow $\PSH(\Xan,\theta)$ 
with the topology of uniform convergence on dual complexes.
Notice that the divisorial points are dense on each dual complex 
$\D_X$, see~\cite[Corollary~3.13]{siminag} or~\cite[Remark~3.9]{jonmus}.
As a consequence of equicontinuity we thus have

\begin{thm}\cite[Theorem~7.8]{siminag}.\label{thm:proper}
 For each model function $\p$ the map $\f\mapsto\sup_X(\f-\p)$ is continuous and proper on $\PSH(X,\theta)$. In particular, the space $\PSH(X,\theta)/\R$ is compact. Further, the topology
  on $\PSH(X,\theta)$ is equivalent to the topology of pointwise convergence on $\Xdiv$.
\end{thm}
Finally we have the following regularization result. Its proof relies on multiplier ideals. 
\begin{thm}\cite[Theorem~8.7]{siminag}.\label{T211}
  For any $\theta$-psh function $\f$, there exists a decreasing net 
  $(\f_j)_j$ of $\theta$-psh model functions that converges 
  pointwise on $X$ to $\f$. 
\end{thm}
The complex analogue of this result is due to Demailly \cite{Dem92} (see also \cite[Appendix]{GZ1} for the case of a line bundle). By Dini's lemma, we get as a consequence:
\begin{cor}{\cite[Corollary 8.8]{siminag}}\label{cor:richberg} The set $\cD(X)\cap\PSH(X,\theta)$ is dense in $C^0(X)\cap\PSH(X,\theta)$ with respect to uniform convergence on $X$. 
\end{cor}

Proposition~\ref{P202} below refines Theorem \ref{T211} and asserts that any $\theta$-psh function is actually the decreasing limit of a \emph{sequence} of $\theta$-psh model functions (but the proof heavily uses Theorem~\ref{T211}).

\subsection{Envelopes}\label{sec:envelopes}
Let $\theta$ be a form as in \S\ref{S301}.
\begin{prop}\cite[Theorem~7.9]{siminag}.\label{prop:uscenv}
If $(\f_\a)_{\a\in A}$ is a family of $\theta$-psh functions that is uniformly bounded above, then the usc upper envelope $(\sup_\a\f_\a)^*$ is also $\theta$-psh. 
\end{prop}
Recall that the usc regularization $u^*$ of a function $u:X\to [-\infty,+\infty[$ is the smallest usc function such that $u^* \ge u$.

\begin{defi}
Let $f: \Xan \to [-\infty,+\infty[$ be any function. We define its \emph{$\theta$-psh envelope} $P_\theta(f)$ as follows. If there does not exist any $\f\in\PSH(X,\theta)$ such that $\f\le f$ on $X$ then we set $P_\theta(f)\equiv -\infty$. Otherwise, we define $P_\theta(f)$ as the usc upper envelope of the set of all $\theta$-psh functions $\f$ such that $\f\le f$ on $X$, \ie we set
$$ 
P_\theta(f):= \left(\sup \left\{\f\mid\f\in\PSH(X,\om),\,\f\le f\right\}\right)^*.
$$
\end{defi}
Thanks to Proposition~\ref{prop:uscenv} $P_\theta(f)$ is either $-\infty$ or belongs to $\PSH(X,\theta)$. If $f$ is usc, then clearly $P_\theta(f)\le f$ on $X$, and $P_\theta(f)$ is then the largest $\theta$-psh function with this property. 

\begin{prop}\cite[Proposition 8.1]{siminag}\label{prop:basicenv} 
\begin{itemize}
\item[(i)] $P_\theta$ is non-decreasing: $f\le g\Rightarrow P_\theta(f)\le P_\theta(g)$. 
\item[(ii)] $P_\theta(f)$ is concave in both arguments:
$$
P_{t\theta+(1-t)\theta'}\left(tf+(1-t)g\right)\ge t P_\theta(f)+(1-t)P_{\theta'}(g)
$$ 
for $0\le t\le 1$.
\item[(iii)] For each $c\in\R$ we have $P_\theta(f+c)=P_\theta(f)+c$.
\item[(iv)] $P_\theta$ is $1$-Lipschitz continuous, \ie $\sup_X|P_\theta(f)-P_\theta(g)|\le\sup_X|f-g|$. 
\item[(v)] Given a bounded function $f$ and a convergent sequence $\theta_m\to\theta$ in $N^1(\cX/S)$ we have $P_{\theta_m}(f)\to P_\theta(f)$ uniformly on $X$. 
\end{itemize}
\end{prop}

% 
%%%%%%%%%%%%%%%%%%%%%%%%%%%%%%%%%%%%%%%%%%%%%%%%%%%%%%%%%%%%%%%%%%%
%
\subsection{Metrized line bundles and curvature forms}\label{S302}
We refer to~\cite{Ch2} for a general account of metrized
line bundles in a non-Archimedean context.
Suffice it to say that a metric $\|\cdot\|$ on a line bundle $L$ on $X$
is a way to produce a local continuous function
$\|s\|$ on (the Berkovich space) $\Xan$ from any local section $s$ of $L$. 

Let $\cX$ be a model and $\cL$ a line bundle on $\cX$ such that 
$\cL|_X=L$. To this data one can associate 
a unique metric $\|\cdot\|_\cL$ on $L$ with the following property:
if $s$ is a nonvanishing local section of $\cL$ on an 
open set $\cU\subset\cX$, then $\|s\|_\cL\equiv1$ 
on $U:=\cU\cap X$. This makes sense since such a section 
$s$ is uniquely defined up to multiplication by an element of 
$\Gamma(\cU,\cO_\cX^*)$ and such elements have norm 1.

More generally, any $\cL\in\Pic(\cX)_\Q$ such that 
$\cL|_X=L$ in $\Pic(X)_\Q$ induces  a metric $\|\cdot\|_\cL$ on 
$L$ by setting $\|s\|_\cL=\|s^{\otimes m}\|_{m\cL}^{1/m}$ for any
$m\in\N^*$ such that $m\cL$ is an actual line bundle. 
Such a metric is called a \emph{model metric} on $L$.

Given a model metric $\|\cdot\|$, any continuous metric on $L$
is of the form $\|\cdot\|e^{-\f}$, with $\f\in C^0(\Xan)$. 
This is a model metric iff $\f$ is a model function.
By a \emph{singular metric} on $L$ we mean an expression of the
form $\|\cdot\|e^{-\f}$ with $\f:\Xan\to[-\infty,+\infty[$ an arbitrary function.

\medskip
Fix a model metric  $\|\cdot\|_\cL$ on $L$ associated to $\cL \in \Pic_\Q(\cX)$. 
The numerical class associated to $\cL$ in $N^1(\cX/S)$ induces a form 
on $\Xan$ in the sense of \S\ref{sec:forms}. 
It does not depend on the choice of model $\cL$ defining the metric.
We call it the \emph{curvature form} of the metric and denote it by $c_1(L,\|\cdot\|)$.
By construction, its de Rham class is given by
\begin{equation}\label{eq:ouf}
  \{ c_1(L,\|\cdot\|) \} = c_1(L) \in N^1(X).
\end{equation}
If $\f\in\cD(\Xan)$ is a model function, then 
\begin{equation*}
  c_1(L, \|\cdot\|\, e^{-\f}) = c_1(L, \|\cdot\|) + dd^c \f,
\end{equation*}
where the form $dd^c\f\in\cZ^{1,1}(\Xan)$ is defined in~\S\ref{sec:forms}. 
\begin{defi}
  Fix a model metric $\|\cdot\|$ on $L$ with curvature form $\theta$. 
  Then a singular metric $\|\cdot\| e^{-\f}$ 
  is \emph{semipositive} if the function $\f$ is $\theta$-psh.
\end{defi}
The results in~\S\ref{S301} have obvious counterparts for singular metrics.
In particular, we have:
\begin{thm}
  Let  $\|\cdot\|$ be a model metric on $L$,  
  associated to a $\Q$-line bundle $\cL$ on a model $\cX$ of $X$.  Then
  \begin{itemize}
  \item[(i)]
    the metric $\|\cdot\|$ is semipositive iff $\cL$ is nef;
  \item[(ii)]
    a continuous metric $\|\cdot\|e^{-\f}$ is semipositive iff 
    there exists a sequence of semipositive model metrics 
    $\|\cdot\|_m = \|\cdot\|e^{-\f_m}$ such that 
    $\f_m \to\f$ uniformly on $\Xan$.
  \end{itemize}
\end{thm}
This result implies that our definition of continuous semipositive metric
coincides with that of Zhang and others.
Unfortunately, the terminology is not uniform across the literature,
see Table~\ref{table} below.

\medskip
\begin{table}[h]
\begin{tabular}{|l| l| }
\hline
\emph{Model metric}: \cite{siminag,yuanzhang}  
& 
\emph{Continuous semipositive metric}:
\\
&  \cite{siminag,Ch1,Ch2}
\\
Algebraic metric: \cite{BPS,Ch1,liu} 
& 
Approachable metric: \cite{BPS}
\\
Smooth metric: \cite{Ch2}
& 
Semipositive metric: \cite{yuanzhang,liu} 
\\
Root of an algebraic metric: \cite{gubler} 
& 
Semipositive admissible metric: \cite{gubler}\\
\hline 
\end{tabular}
\caption{Terminology for metrics on line bundles.}\label{table}
\end{table}

% 
%%%%%%%%%%%%%%%%%%%%%%%%%%%%%%%%%%%%%%%%%%%%%%%%%%%%%%%%%%%%%%%%%%% 
%
\subsection{Intersection numbers and Monge-Amp{\`e}re measures}\label{S109}
The Monge-Amp{\`e}re operator that we will use arises from intersection theory
on models.

Let $\cX$ be a model of $X$, and pick 
numerical classes $\theta_{1,\cX},\dots,\theta_{n,\cX} \in N^1(\cX/S)$.
For any vertical divisor $D \in \Div_0(\cX)$ we  define
\begin{equation*}
  D\cdot \theta_1 \cdot\ldots\cdot\theta_n :=
  \sum_E \ord_E(D) \, (\theta_{1,\cX}|_E \cdot\ldots\cdot\theta_{n,\cX}|_E),
\end{equation*}
where $E$ ranges over all irreducible components of the special fiber $\cX_0$.
We obtain a pairing that is linear in each entry and symmetric in the $\theta_i$'s.
\begin{prop-def}
  To any $n$-tuple $(\theta_1,\dots,\theta_n)$ of closed $(1,1)$-forms
  we can associated a signed atomic measure
  $\theta_1\wedge\dots\wedge \theta_n$ supported on $\Xdiv$ such that 
  \begin{equation}\label{eq:intersect-forms}
    \int_{\Xan}  f \, \theta_1\wedge\dots\wedge \theta_n
    = 
    \sum_{i\in I} b_i f(x_i)\, (\theta_{1,\cX}|_{E_i} \cdot\ldots\cdot\theta_{n,\cX}|_{E_i})
  \end{equation}
  for any common determination $\cX$ of the forms $\theta_i$, 
  and for any model function $f$.
  Here we have written the special fiber as $\cX_0=\sum_{i\in I}b_iE_i$
  and $x_i=x_{E_i}$ is the divisorial point associated to $E_i$. 
  
  Further, $(\theta_1,\dots, \theta_n) \mapsto \theta_1 \wedge\dots\wedge \theta_n$ 
  is multilinear and symmetric. 
\end{prop-def}
\begin{proof}
  Choose a common determination of the forms $\theta_i$, and define 
  $\int_{\Xan}  f \, (\theta_1\wedge\dots\wedge \theta_n)$  
  using~\eqref{eq:intersect-forms}.
  The fact that $\int_{\Xan}  f \, (\theta_1\wedge\dots\wedge \theta_n)$ 
  does not depend on the choice of a determination $\cX$ is a consequence of 
  the projection formula
  \begin{equation*}
    \pi_* D \cdot \theta_{1,\cX} \cdot\ldots\cdot\theta_{n,\cX}
    =D \cdot \pi^* \theta_{1,\cX} \cdot\ldots\cdot\pi^* \theta_{n,\cX}
  \end{equation*}
  if $\pi:\cX' \to \cX$, and $D$ is any vertical divisor in $\cX'$.
  
  Then by construction $\theta_1 \wedge\dots\wedge \theta_n$ 
  can be identified with the atomic measure  $\sum_i w_i \delta_{x_i}$ 
  with $w_i = (\theta_{1,\cX}|_{E_i} \cdot\ldots\cdot\theta_{n,\cX}|_{E_i})$. 
  This measure is supported on the divisorial points associated 
  to the irreducible components of $\cX_0$. The last statement is clear.
\end{proof} 
  
  \begin{prop}\label{P210}
    If the forms $\theta_1,\dots,\theta_n$ are semipositive, 
    then $\theta_1 \wedge\dots\wedge \theta_n$ is a positive measure, of mass 
    \begin{equation}\label{e301}
      \int_{\Xan} \theta_1\wedge\dots\wedge \theta_n = \{ \theta_1\} \cdot\ldots\cdot \{\theta_n\}.
    \end{equation}
  \end{prop}
  \begin{proof} Pick a model $\cX$ such that each $\theta_i$ is determined by a nef class $\theta_{i,\cX}\in N^1(\cX/S)$. The restriction of $\theta_{i,\cX}$ to each component $E_\om$ of $\cX_0$ is then also nef, and it follows that the intersection number $(\theta_{1,\cX}|_E\cdot...\cdot\theta_{n,\cX}|_E)$ is non-negative, hence the first assertion. Since the constant function $1$ corresponds to the vertical divisor $\cX_0$ we have by definition
  $$
  \int_X\theta_1\wedge...\wedge\theta_n=\cX_0\cdot\theta_1\cdot\ldots\cdot\theta_n.
  $$
 By~\cite[Example~20.3.3]{FultonInter} this is the same as 
    the intersection number against the generic fiber of $\cX$,
    and this is equal to $\{ \theta_1\} \cdot\ldots\cdot \{\theta_n\}$ by definition. \end{proof}

As a special case, fix $\theta\in\cZ^{1,1}(\Xan)$.
To any $\theta$-psh model functions $\f_1,\dots,\f_n$ we then associate a
\emph{mixed Monge-Amp{\`e}re measure}
\begin{equation*}
  (\theta+dd^c\f_1)\wedge\dots\wedge(\theta+dd^c\f_n).
\end{equation*}
This is an atomic positive measure on $\Xan$ of mass $\{\theta\}^n$.

Analogously to the complex case 
we have the following \emph{integration by parts} formula:
\begin{prop}\label{P303}
  If $f,g\in\cD(\Xan)$ are model functions and $\theta_1,\dots,\theta_{n-1}$ are closed $(1,1)$-forms then  we have 
  \begin{equation*}
    \int f\,dd^c g\wedge\theta_1\wedge\dots\wedge\theta_{n-1}
    =\int g\,dd^c f\wedge\theta_1\wedge\dots\wedge\theta_{n-1}.
  \end{equation*}
  
\end{prop}
\begin{proof}
Pick a common determination $\cX$ of $f, g$ and the $\theta_i$'s, and 
divisors $D,D', D_i$ such that  $f = \f_D$, $g= \f_{D'}$ and $\theta_i$ is the class
in $N^1(\cX/S)$ induced by $D_i$. Then by definition we have
\begin{multline*}
\int f\,dd^c g\wedge\theta_1\wedge\dots\wedge\theta_{n-1}
= \sum_E \ord_E(D)  (D'|_E \cdot D_1|_E \cdot ... \cdot D_{n-1}|_E)
\\
= \sum_{E,E'} \ord_E(D) \ord_{E'}(D')\,   (D_1|_E)|_{E'\cap E} \cdot ... \cdot (D_{n-1}|_E)|_{E'\cap E}
\\
= \sum_{E,E'} \ord_E(D) \ord_{E'}(D')\,   (D_1|_{E'})|_{E\cap E'} \cdot ... \cdot (D_{n-1}|_{E'})|_{E\cap E'}
\\
=
\int g\,dd^c f\wedge\theta_1\wedge\dots\wedge\theta_{n-1}
\end{multline*}
where the third equality follows from~\cite[Theorem 2.4]{FultonInter}.
\end{proof}

The next result follows from the Hodge index theorem, compare~\cite[Theorem 2.1.1]{yuanzhang}.
\begin{prop}\label{prop:cauchys}
  Suppose $\theta_1,\dots,\theta_{n-1}$ are semipositive closed $(1,1)$-forms.
  Then the symmetric bilinear form
  \begin{equation*}
    (f, g) \mapsto 
    \int_{\Xan} f\,dd^c g \wedge\theta_1\wedge\dots\wedge\theta_{n-1}
  \end{equation*}
  on $\cD(\Xan)$ is negative semidefinite. In particular, for any two model functions $f$, $g$, the following Cauchy-Schwarz inequality holds: 
\begin{multline}
  \left|\int_{\Xan} f\,dd^c g\wedge\theta_1\wedge\dots\wedge\theta_{n-1}\right|
  \le \\
  \left(-\int_{\Xan} f\,dd^c f \wedge\theta_1\wedge\dots\wedge\theta_{n-1}
  \right)^{1/2}\,
  \left(-\int_{\Xan} g\, dd^c g \wedge\theta_1\wedge\dots\wedge\theta_{n-1}
  \right)^{1/2}. \label{eq:CS}
\end{multline}
\end{prop}

\begin{proof}[Proof of Proposition~\ref{prop:cauchys}]
  Fix a model function $f$. We need to prove
  \begin{equation*}
    I:=\int_{\Xan} f\,dd^c f \wedge\theta_1\wedge\dots\wedge\theta_{n-1}\le0
  \end{equation*}
  Choose a common determination $\cX$ of $\f$ and all the $\theta_i$.
  By continuity, we may assume $\f = \f_D$ for some $D \in \Div_0(\cX)_\Q$, 
  and each form $\theta_i$ is determined by a $\Q$-line bundle $\cL_i$ on $\cX$. 
  Then $I = D^2\cdot \cL_1 \cdot\ldots\cdot \cL_{n-1}$ and 
  the result  follows from~\cite[Theorem 2.1.1 (a)]{yuanzhang}.
\end{proof}

\begin{rmk} In the complex case we have by Stokes' theorem
$$
\int f\,dd^c g\wedge\theta_1\wedge...\wedge\theta_{n-1}=-\int df\wedge d^c g\wedge\theta_1\wedge...\wedge\theta_{n-1},
$$
and negativity comes from that of the $(1,1)$-form $df\wedge d^c f$. Recall also that 
$df\wedge d^c f\wedge\om^{n-1}=|df|_\om^2\,\om^n$ 
when $\om$ is a K\"ahler form, so that $\left(-\int f\,dd^c f\wedge\om^{n-1}\right)^{1/2}$ is the $L^2$-norm of the gradient of $f$.
\end{rmk}

%
%%%%%%%%%%%%%%%%%%%%%%%%%%%%%%%%%%%%%%%%%%%%%%%%%%%%%%%%%%%%%%%%%%%
%
\subsection{Radon measures and convergence results} 
We shall make frequent use of basic integration and measure theory.
Let $X$ be a compact (Hausdorff) space.
A \emph{Radon measure} on $X$ is a positive 
linear functional $\mu:C^0(X)\to\R$. 
With this definition, it follows from the Riesz representation 
theorem that Radon measures are in 1-1 correspondence with 
regular Borel measures on $X$; see~\cite[\S7.1--2]{Folland}.

Since we shall be dealing with (possibly uncountable)
nets rather than sequences, one has to be careful using results 
from integration theory. For example, the
monotone convergence theorem is of course not true for general nets.
However, as the next results show, 
integration of semicontinuous functions against Radon measures
is often well behaved.
\begin{lem}~\cite[Proposition~7.12]{Folland}.\label{L301}
  If $\mu$ is a positive Radon measure on $X$ 
  and $(f_j)_j$ a decreasing net of usc functions on $X$,
  converging pointwise to a (usc) function $f$, then 
  $\lim_j\int f_j\mu=\int f\mu$.
\end{lem}
In particular, one has
\begin{lem}\label{L207}~\cite[Corollary~7.13]{Folland}.
  If $\mu$ is a positive Radon measure on $X$ 
  and $f$ is a usc function on $X$, then 
  \begin{equation*}
    \int f\mu=\inf\left\{\int g\mu\mid f\le g,\,g\in C^0(X)\right\}
  \end{equation*}
\end{lem} 
\begin{cor}\label{C203}
  Let $(f_j)_j$ a decreasing 
  net of usc functions on $X$ converging pointwise to a (usc) function $f$, 
  and $(\mu_j)_j$ a net of positive Radon measures on $X$ 
  converging weakly to a positive  Radon measure $\mu$. 
  Then 
  \begin{equation*}
    \limsup_j\int f_j\mu_j\le\int f\mu.
  \end{equation*}
\end{cor}
\begin{proof} 
  Upon replacing $\mu_j$ with $(\int\mu_j)^{-1}\mu_j$ 
  we may assume that the $\mu_j$'s are probability measures. 
  Fix any $\e>0$. 
  By Lemma~\ref{L207} there exists a continuous
  function $g\ge f$ on $X$ such that 
  $\int g\mu<\int f\mu+\e$.
  By Dini's lemma, we have $f_j<g+\e$ for all $j\gg1$, hence 
  \begin{equation*}
    \limsup_j\int f_j\mu_j
    \le\limsup_j\int g\mu_j+\e
    =\int g\mu+\e
    \le\int f\mu+2\e.
  \end{equation*}
  since $\int g\mu_j\to\int g\mu$ by the definition of 
    weak convergence. The result follows.
\end{proof}
% 
%
%%%%%%%%%%%%%%%%%%%%%%%%%%%%%%%%%%%%%%%%%%%%%%%%%%%%%%%%%%%%%%%%%%%
%
%
%\newpage
\section{Monge-Amp\`ere operator on bounded functions}\label{S102}
From now on we fix a form $\theta\in\cZ^{1,1}(\Xan)$ whose de Rham class $\{\theta\}\in N^1(X)$ is \emph{ample}. 
In the next three sections we shall develop some of the Bedford-Taylor theory
in our non-Archimedean setting. 

Our first main objective is to extend the Monge-Amp{\`e}re operator
defined in~\S\ref{S109} from $\theta$-psh model functions to bounded 
$\theta$-psh functions. 
\begin{thm}\label{T202} 
  There exists a unique operator 
  \begin{equation*}
    (\f_1,\dots,\f_n)\mapsto(\theta+dd^c\f_1)\wedge\dots\wedge(\theta+dd^c\f_n)
  \end{equation*}
  taking an $n$-tuple of \emph{bounded} $\theta$-psh functions to a 
  positive Radon measure on $X$ of mass $\{\theta\}^n$
   and such that 
  \begin{itemize}
  \item the definition is compatible with the  definition for $\theta$-psh model functions given in~\S\ref{S109}; 
  \item for any decreasing nets of bounded $\theta$-psh functions $\p^j \to\p$, 
 and $\f_i^j \to \f_i$ for $i=1,\dots, n$ we have
    \begin{equation*}
    \int \p^j \, (\theta+dd^c\f_1^j)\wedge\dots\wedge(\theta+dd^c\f_n^j)
    \longrightarrow
    \int \p \, (\theta+dd^c\f_1)\wedge\dots\wedge(\theta+dd^c\f_n). 
  \end{equation*}
  \end{itemize}
\end{thm}
\begin{rmk}
  One can also prove continuity 
  along \emph{increasing} nets but we will not need this.
\end{rmk}

Note that the uniqueness part of Theorem~\ref{T202} follows from the 
fact that any $\theta$-psh function is the decreasing limit of a net of 
$\theta$-psh model functions, see Theorem~\ref{T211}. 
For the same reason, the mapping  
\begin{equation*}
  (\f_1,\dots, \f_n) \mapsto  (\theta+dd^c\f_1)\wedge\dots\wedge(\theta+dd^c\f_n)
\end{equation*}
is symmetric in its arguments, and additive in the following sense:
\begin{multline*}
  (\theta+t dd^c\f_1+ (1-t) dd^c \f_1')\wedge(\theta+dd^c\f_2)
  \wedge\dots\wedge(\theta+dd^c\f_n) =\\
  t \, (\theta+dd^c\f_1)\wedge\dots\wedge(\theta+dd^c\f_n) + (1-t)\,
  (\theta+dd^c\f'_1)\wedge\dots\wedge(\theta+dd^c\f_n)
\end{multline*}
for $0\le t\le 1$. This additivity property in particular implies
\begin{equation}\label{eqMAmiddle}
  \left(\theta +  dd^c ( t\f + (1-t) \p)\right)^{n}
  \ge
  t^n (\theta + dd^c \f)^n +  (1-t)^n (\theta + dd^c \p)^n
\end{equation}
in the sense of measures, for all bounded $\theta$-psh 
functions $\f,\p$, and any $ 0 \le t \le 1$.

Given bounded $\theta$-psh functions,
one can now define signed measures 
\begin{equation*}
  dd^c\f_1\wedge\dots\wedge dd^c\f_p\wedge(\theta + dd^c \f_{p+1})
  \wedge\dots\wedge(\theta+dd^c\f_n)
\end{equation*}
by writing $dd^c \f_i = (\theta+ dd^c \f_i) - \theta$ and expanding the product 
formally using multilinearity. These products are also continuous along 
decreasing nets, and we thus obtain
\begin{cor}\label{C202}
  If $\f_1,\dots,\f_{n-1}$ are bounded $\theta$-psh functions on $X$, then
  the bilinear form 
  \begin{equation*}
    (\f,\p)\mapsto
    \int(-\f)\, dd^c\p\wedge(\theta+dd^c\f_1)\wedge\dots\wedge(\theta+dd^c\f_{n-1})
  \end{equation*}
  is well-defined and positive semidefinite on the vector space spanned
  by the set of bounded $\theta$-psh functions.
\end{cor}
In particular, the Cauchy-Schwarz inequality~\eqref{eq:CS} holds for all bounded 
$\theta$-psh functions $\f_2,\dots, \f_n$ and for all functions $\p,\f$ 
that are differences of bounded $\theta$-psh functions.

% 
%%%%%%%%%%%%%%%%%%%%%%%%%%%%%%%%%%%%%%%%%%%%%%%%%%%%%%%%%%%%%%%%%%%
%
\subsection{Proof of Theorem~\ref{T202}}
We adapt to our setting the  Bedford-Taylor 
approach as explained,  for instance, in~\cite[Theorem~3.7, p.188]{DemBook}. 

Fix $ 0 \le p \le n$ and $\theta$-psh \emph{model functions} 
$\f_{p+1}',\dots,\f_n'$. Consider the following statement.

\begin{assertAp}
  To any $p$-tuple $\f_1,\dots,\f_p$ of bounded $\theta$-psh 
  functions is  associated a positive Radon measure $\MAC(\f_1,\dots, \f_p)$ of mass $\{\theta\}^n$ 
  such that:
  \begin{itemize}
  \item 
    if $\f_1,\dots,\f_p$ are model functions then 
    \begin{equation}\label{eq:int-by-parts}
      \MAC (\f_1,\dots,\f_p)
      =(\theta+dd^c\f_1)\wedge\dots\wedge(\theta+dd^c\f_p)
      \wedge(\theta+dd^c\f_{p+1}')\wedge\dots\wedge(\theta+dd^c\f_n')
    \end{equation}
  \item 
    the mapping
    \begin{equation*}
      (\p,\f_1,\dots,\f_p)\mapsto\int\p\MAC(\f_1,\dots,\f_p)
    \end{equation*}
    is continuous along decreasing nets of bounded $\theta$-psh functions.
  \end{itemize}
\end{assertAp}

We shall prove A$(p)$  by induction on $p$.
Observe that for $p=n$, this proves Theorem~\ref{T202}.

\medskip
The assertion A$(0)$ is clear, since 
$\MAC(\f_1',\dots,\f_n')$ 
is a finite sum of Dirac masses at divisorial points of $\Xan$. 
Assume that A$(p-1)$ holds for any $(n-p+1)$-tuple of $\theta$-psh 
model functions and let 
$\f_{p+1}',\dots,\f_n'$ be $\theta$-psh model functions. 

Given bounded $\theta$-psh functions $\f_1,\dots,\f_p$, 
we define $\MAC(\f_1,\dots,\f_p)$ by forcing the integration by parts formula
\begin{multline*}\label{eq:byparts}
\tag{\dag}  \int\p\MAC(\f_1,\dots,\f_{p-1},\f_p)
  := \\
 \int\f_p\, (\theta +dd^c\f_1)
 \wedge\dots\wedge(\theta+ dd^c \f_{p-1})\wedge (\theta+ dd^c\p) \wedge (\theta + \f'_{p+1})
 \wedge\dots\wedge(\theta+ dd^c \f'_n) \\
 +\int(\p-\f_p)
 (\theta +dd^c\f_1)
 \wedge\dots\wedge(\theta+ dd^c \f_{p-1})\wedge\theta\wedge(\theta + \f'_{p+1})
 \wedge  \dots\wedge (\theta+ dd^c \f'_n)
\end{multline*}
for every model function $\p$. 

\smallskip
Observe that 
the right-hand side is continuous along
decreasing nets as a function of $(\f_1,\dots,\f_p)$ by the induction
hypothesis. Since equality holds in~\eqref{eq:byparts} when all  the $\f_i$ are model
functions and since $\MAC(\f_1,\dots, \f_p)$ is a positive measure of mass $\{ \theta\}^n$, 
it follows by regularization (Theorem~\ref{T211}) that the right-hand side 
is also linear in $\psi$, and non-negative when $\psi \ge0$.

Now the space of model functions is spanned by 
$\theta$-psh model functions by Proposition~\ref{Pgenerators};
hence $\MAC(\f_1,\dots,\f_p)$ is well-defined as a positive measure of mass $\{\theta\}^n$ 
and is continuous along decreasing nets as a function 
of $(\f_1,\dots,\f_p)$. It remains to show that 
\begin{equation*}
  (\p,\f_1,\dots,\f_p)\mapsto\int\p\MAC(\f_1,\dots,\f_p)
\end{equation*}
is continuous along decreasing nets of bounded 
$\theta$-psh functions. Let thus $(\f_i^j)_j$, $i=1,\dots,p$ 
and $\p^j$ be decreasing nets 
of $\theta$-psh functions converging, respectively, to 
bounded $\theta$-psh functions $\f_i$ and $\p$. Set 
\begin{equation*}
  \mu^j:=\MAC(\f_1^j,\dots,\f_p^j).
\end{equation*}
We already know that $\mu^j$ converges weakly to
$\mu:=\MAC(\f_1,\dots,\f_p)$. Since $\p^j$ is usc for each $j$,
Corollary~\ref{C203} yields
\begin{equation*}
  \limsup_j\int\p^j\mu^j\le\int\p\mu.
\end{equation*} 
For the reverse estimate, we rely on the 
following approximate monotonicity property:
\begin{lem}\label{lem:mono} 
  Let $\p$ and $\chi_i\ge\f_i$, $i=1,\dots,p$ 
  be bounded $\theta$-psh functions. Then we have 
  \begin{align*}
    \int\p\MAC(\chi_1,\dots,\chi_p)
    &\ge\int\p\MAC(\f_1,\dots,\f_p)\\
    &+\sum_{i=1}^p\int(\f_i-\chi_i)\MAC(\f_1,\dots,\f_{i-1},0,\chi_{i+1},\dots,\chi_p).
 \end{align*}
\end{lem}
The lemma implies that, for each $j$:
  \begin{equation*}
    \int\p^j\mu^j
    \ge
    \int\p\mu^j
    \ge
    \int\p\mu
    +\sum_{i=1}^p\int(\f_i-\f_i^j)\MAC(\f_1,\dots,\f_{i-1},0,\f_{i+1}^j,\dots,\f_p^j).
  \end{equation*}
  By the inductive hypothesis, the sum in the right-hand side 
  tends to $0$ as $j\to\infty$, so we infer as desired that 
  $\liminf_j\int\p^j\mu^j\ge\int\p\mu$. 
\begin{proof}[Proof of Lemma~\ref{lem:mono}]
  Note first that $\p$ may be assumed to be a model function by
  \begin{lem}\label{lem:inf} 
    Let $\nu$ be a positive Radon measure on $X$
    and let $\f$ be a bounded $\theta$-psh function. 
    Then we have
    \begin{equation*}
      \int\f\nu=\inf_{\p\ge\f}\int\p\nu
    \end{equation*}
    where $\p$ ranges over all
    $\theta$-psh \emph{model functions} such that $\p\ge\f$.
  \end{lem}
 Since we already know that $(\f_1,\dots,\f_p) \mapsto\MAC(\f_1,\dots,\f_p)$ 
  is continuous along decreasing nets, we may by regularization
  assume that all $\f_i$ and $\chi_i$ are also model functions.
  Integration by parts~\eqref{eq:byparts} then yields
  \begin{multline*}
    \int\p\MAC(\chi_1,\chi_2,\dots,\chi_p)-\int\p\MAC(\f_1,\chi_2,\dots,\chi_p)=\\
    =\int(\chi_1-\f_1)\MAC(\p,\chi_2,\dots,\chi_p)-\int(\chi_1-\f_1)\MAC(0,\chi_2,\dots,\chi_p)
  \end{multline*}
  hence 
  \begin{equation*}
    \int\p\MAC(\chi_1,\dots,\chi_p)
    \ge\int\p\MAC(\f_1,\chi_2,\dots,\chi_p)
    +\int(\f_1-\chi_1)\MAC(0,\chi_2,\dots,\chi_p). 
  \end{equation*}
  We similarly have
  \begin{align*}
    \int\p\MAC(\f_1,\chi_2,\chi_3,\dots,\chi_p)
    &\ge\int\p\MAC(\f_1,\f_2,\chi_3,\dots,\chi_p)\\
    &+\int(\f_2-\chi_2)\MAC(\f_1,0,\chi_3,\dots,\chi_p). 
  \end{align*}
  Iterating this argument and summing up then yields the desired result. 
\end{proof}

\begin{proof}[Proof of Lemma~\ref{lem:inf}]
  Let $\e>0$. Since $\f$ is usc, Lemma~\ref{L207}
  shows that there exists a continuous function 
  $v$ on $\Xan$ such that $v\ge\f$ and
  $\int v\nu\le\int\f\nu+\e$. 
  The result now follows since~\cite[Corollary~8.6]{siminag}
  yields an $\theta$-psh model function $\psi$ such that 
  $\f\le\psi\le v+\e$. 
\end{proof}
\begin{defi}
	A \emph{pluripolar set} is a subset of $\{\p=-\infty\}$ for some
	$\p\in\PSH(X,\theta)$.  
\end{defi}
\begin{prop}\label{prop:pshint} 
	Let $\f_1,...,\f_n$ be bounded $\theta$-psh functions. Then any $\p\in\PSH(X,\theta)$
	is integrable with respect to the measure 
	$\mu:=(\theta+dd^c\f_1)\wedge\cdots\wedge(\theta+dd^c\f_n)$.
	In particular, $\mu$ does not put mass on pluripolar sets.
\end{prop}

\begin{proof} 
 Pick $\f_0\in\PSH(X,\theta)\cap\cD(X)$. 
 Upon replacing $\theta$, $\f_i$, and $\p$
 with $\theta+dd^c\f_0$, $\f_i-\f_0$ and $\p-\f_0$ respectively, 
 we may assume that $\theta$ is semipositive and that $\f_i\le 0$
 for all $i$. 
 Adding a constant to $\p$ we may also assume $\sup_X\p=0$. 
 Set $M:=\max_i\sup|\f_i|$.
 First assume that $\p$ is also bounded. We claim that
 $\int-\p\mu$ is bounded by
 a constant depending only on $M$ (but not on $\sup_X|\p|$). 
 Integrating by parts we have
 \begin{multline*}
   0\le\int(-\p)\mu
   =\int(-\p)\theta\wedge(\theta+dd^c\f_2)\wedge\dots\wedge(\theta+dd^c\f_n)\\
   +\int(-\f_1)(\theta+dd^c\p)\wedge(\theta+dd^c\f_2)\wedge\dots\wedge(\theta+dd^c\f_n)
   +\int\f_1\theta\wedge(\theta+dd^c\f_2)\wedge\dots\wedge(\theta+dd^c\f_n).
 \end{multline*}
 Here the second to last integral is bounded by $M\{\theta\}^n$, 
 while the last integral to the right is non-positive since 
 $\theta\wedge(\theta+dd^c\f_2)\wedge\dots\wedge(\theta+dd^c\f_n)$ 
 is a positive measure. Hence
 \begin{equation*}
   \int(-\p)\mu
   \le\int(-\p)\theta\wedge(\theta+dd^c\f_2)\wedge\dots\wedge(\theta+dd^c\f_n)
   +M\{\theta\}^n.
 \end{equation*}
 Iterating this argument yields
 \begin{equation*}
   0\le\int(-\p)\mu\le\int(-\p)\theta^n+n M\{\theta\}^n. 
 \end{equation*}
 Now $\int(-\p)\theta^n$ is bounded above by some $C>0$ only 
 depending on $\theta$, by compactness of 
 $\{\p\in\PSH(X,\theta)\mid\sup_X\p=0\}$ and the fact that $\theta^n$ 
 is an atomic measure supported at finitely many divisorial points. 
 We conclude that 
 \begin{equation}\label{eq:l1psh}
   0\le\int(-\p)\mu\le C+nM\{\theta\}^n
 \end{equation} 
 for some constant $C>0$ only depending on $\theta$, as long as $\p$
 is a bounded $\theta$-psh function with $\sup_X\p=0$. 
 If $\p$ is now a possibly unbounded $\theta$-psh function 
 normalized by $\sup_X\p=0$, $\p$ is the decreasing limit of the 
 bounded $\theta$-psh functions $\p_m:=\max\{\p,-m\}$, 
 so that (\ref{eq:l1psh}) continues to hold, by monotone convergence. 
\end{proof}

%
%%%%%%%%%%%%%%%%%%%%%%%%%%%%%%%%%%%%%%%%%%%%%%%%%%%%%%%%%%%%%%%%%%%
%
\subsection{The Chambert-Loir measure}
We follow the notation and terminology of~\S\ref{S302}.
Consider an ample line bundle $L$ on $X$ and equip $L$ with a 
model metric $\|\cdot\|$. 
Any continuous metric on $L$ is then of the form 
$\|\cdot\|\, e^{-\f}$ where $\f\in C^0(X)$. Recall that
this metric is semipositive iff the function $\f$ is $\theta$-psh,
where $\theta:=c_1(L,\|\cdot\|)$.
In this case, set 
\begin{equation*}
  c_1(L, \|\cdot\| e^{-\f})^n:=(\theta+dd^c\f)^n,
\end{equation*}
where the right hand side is the positive Radon measure in Theorem~\ref{T202}.

This is the same measure as the one defined by
Chambert-Loir in~\cite{Ch1}. 
Indeed, this is certainly true when $\f$ 
is a model function, as seen by comparing~\eqref{eq:intersect-forms}
and~\cite[D{\'e}finition~2.4]{Ch1}.
In general, Corollary~\ref{cor:richberg} yields a sequence $(\f_m)_{m=1}^\infty$ 
of $\theta$-psh model functions converging uniformly to $\f$
on $\Xan$. 
The measure $\mu$ associated to $(L,\|\cdot\|e^{-\f})$ by Chambert-Loir
is the limit of the measures $\mu_m:=(\theta+dd^c\f_m)^n$,
see~\cite[Proposition~2.7]{Ch1}. 
But after replacing $\f_m$ by $\f_m+\e_m$ with a suitable 
sequence $\e_m\searrow0$,
we may assume that the sequence $\f_m$ is decreasing, 
hence $\mu=(\theta+dd^c\f)^n$ by Theorem~\ref{T202}.
%
%
%%%%%%%%%%%%%%%%%%%%%%%%%%%%%%%%%%%%%%%%%%%%%%%%%%%%%%%%%%%%%%%%%%%
%
%
\section{Capacity and quasicontinuity}\label{S103}
Let $\om\in\cZ^{1,1}(\Xan)$ be a closed $(1,1)$-form
with ample de Rham class $\{\om\}\in N^1(X)$. It is convenient to assume that $\{\om\}^n=1$, a harmless assumption by homogeneity. Let us further assume from now on that $\om$ is \emph{semipositive},
that is, $\R\subset\PSH(X,\om)$. 

In this section, we introduce a capacity
that will be used to measure the size of subsets of $\Xan$.
It is the analogue of the Monge-Amp{\`e}re capacity introduced
in~\cite{BT1} and adapted to the case of compact K\"ahler 
manifolds in~\cite{GZ1}.

The Monge-Ampere operator of course also depends on the choice of $\omega$ but we write 
\begin{equation*}
  \MA(\f_1,\dots, \f_n ):= (\om + dd^c \f_1) \wedge\dots\wedge(\om+ dd^c \f_n)
\end{equation*}
as well as $\MA(\f):=\MA(\f,\dots, \f)=(\om+dd^c\f)^n$ to simplify some of the formulas below. 

\begin{defi}
  For any Borel set $E\subseteq\Xan$, set 
  \begin{equation*}
  \Capa_\om(E)=\sup\left\{\int_E\MA(\f)\mid\f\in\PSH(X,\om),\, -1\le \f\le 0\right\}.
  \end{equation*}
\end{defi}
By Proposition \ref{P210} we have $0\le\Capa_\om(E)\le\{\om\}^n$. Note that if $E_1, E_2,\dots$ are Borel sets, then 
$\Capa_\om(\bigcup E_j)\le\sum_j\Capa_\om(E_j)$.

The Monge-Amp{\`e}re operator and the capacity of course 
depend on the choice of $\om$, but we drop this dependence for notational simplicity. 

\begin{lem}\label{L203}
  If $x\in X$ is a divisorial point, then $\Capa_\om\{x\}>0$. As a 
  consequence, every nonempty open subset of $\Xan$
  has strictly positive capacity.
\end{lem}
\begin{proof}
  The second statement follows from the first since
  divisorial points are dense in $\Xan$, see~\S\ref{S303}.
  To prove the first statement, pick an SNC model
  $\cX$ of $X$ such that $x=x_E$ is associated to an
  irreducible component $E_\om$ of the special fiber.
  By~\cite[Proposition~5.2]{siminag} there exists $u\in\cD(X)$ 
  determined on $\cX$ 
  such that  $-1\le u\le 0$ and $\om+dd^cu$ is determined by  an \emph{ample} class in $N^1(\cX/S)$.
  Then $\Capa_\om\{x\}\ge\MA(u)\{x\}=b_E((\om+dd^cu)|_E)^n>0$,
  see~\S\ref{S109}.
\end{proof}
The next two propositions are the main results of this section.
\begin{prop}\label{P201}
  If $\f$ is a bounded $\om$-psh function then for each $\e>0$
  there exists an open subset $G\subseteq\Xan$ with $\Capa_\om(G)<\e$
  and a decreasing sequence $(\f_m)_{m=1}^\infty$ of 
  $\om$-psh model functions that converges uniformly to 
  $\f$ on $G^c$. In particular, $\f$ is continuous on $G^c$. 
\end{prop}
\begin{defi}\label{defi:quasicont}
A function $h: \Xan \to \R$ is said to be \emph{quasicontinuous} iff it is
continuous outside sets of arbitrarily small capacity. 
\end{defi}
The previous result can be thus rephrased by saying  that bounded $\om$-psh functions are 
quasicontinuous.

\smallskip
Using the same technique we shall replace nets by sequences in the 
regularization result for $\om$-psh functions (Theorem~\ref{T211}). 
While not crucial, this result is psychologically satisfying and 
does simplify the proof of Corollary~\ref{C112} below.
\begin{prop}\label{P202}
  Any $\om$-psh function $\f$ is the limit of a decreasing 
  sequence $(\f_m)_{m=1}^\infty$ of $\om$-psh model 
  functions.
\end{prop}
The rest of this section is devoted to the proof of these two propositions.
First we state and prove two estimates on special Monge-Amp\`ere integrals.
\begin{lem}\label{L201} The Monge-Amp\`ere measure of any bounded $\om$-psh is linearly bounded by the capacity. More precisely, if $u$ is an $\om$-psh function such that $-M\le u\le 0$,
  where $M\ge 1$, then 
  \begin{equation*}
    \MA(u)\le M^n\Capa_\om
  \end{equation*}
  on Borel sets. 
\end{lem}
\begin{proof} Given a Borel set $E\subset X$ we have
  \begin{equation*}
  \int_E\MA(u)
  =\int_E(\om+dd^cu)^n
  \le\int_E(M\om+dd^cu)^n
  = M^n\int_E(\om+dd^c\frac{u}{M})^n
  \le M^n\Capa_\om(E).
  \end{equation*}
  Here the first inequality follows by writing
  $M\om + dd^c u = (M-1)\om + \om + dd^cu$ 
  and expanding the Monge-Amp\`ere measure
  by multilinearity.
\end{proof}
\begin{lem}\label{L202}
  Suppose $\f$, $\p$ and $\f_1,\dots,\f_n$ are bounded
  $\om$-psh functions such that $-M\le\f\le\p\le0$
  and $-M\le u_i\le 0$, where $M\ge1$. Then 
  \begin{equation*}
  0\le\int(\p-\f)\MA(\f_1,\dots,\f_n)\le 4M
  \left(\int(\p-\f)\MA(\frac\f2)\right)^\frac{1}{2^n}.
  \end{equation*}
\end{lem}
\begin{proof}
  After regularizing we may assume that all  functions involved are model functions.
  Write, symbolically, $T=(\om+dd^c\f_2)\wedge\dots\wedge(\om+dd^c\f_n)$.
  Then
  \begin{equation*}
    \int(\p-\f)\MA(\f_1,\dots,\f_n)
    =\int(\p-\f)\om\wedge T
    +\int(\p-\f)\,dd^c\f_1\wedge T.
  \end{equation*}
   Since $0\le\int(\p-\f)\om\wedge T\le M$, the first term in the right-hand side satisfies 
  \begin{equation*}
    \int(\p-\f)\om\wedge T
    \le M^{\frac12}\left(\int(\p-\f)\om\wedge T\right)^{\frac12}.
  \end{equation*}
  By the Cauchy-Schwarz inequality (Corollary~\ref{C202}), the second term
is bounded by 
  \begin{equation*}
    \left(\int(\p-\f)\,dd^c(\f-\p)\wedge T\right)^{\frac12}
    \left(\int(-\f_1)\,dd^c\f_1\wedge T\right)^{\frac12}
  \end{equation*}
  By the assumption that $-M\le u_1\le 0$ and $\int\om^n=1$ we have
  \begin{equation*}
    0\le\int(-\f_1)\,dd^c\f_1\wedge T
    =\int \f_1\om\wedge T-\int \f_1(\om+dd^c\f_1)\wedge T
    \le M.
  \end{equation*}
  Similarly,
  \begin{align*}
    0\le\int(\p-\f)\, dd^c (\f-\p) \wedge T
    &=\int(\p-\f)(\om+dd^c\f)\wedge T-\int(\p-\f)(\om+dd^c\p)\wedge T
    \\ & \le \int(\p-\f)(\om+dd^c\f)\wedge T.
  \end{align*}
  Putting this together, and using the concavity of the square root, 
  we get
  \begin{multline*}
    \int(\p-\f)\MA(\f_1,\dots,\f_n)
    \le \\ M^{\frac12}\left(
      \left(\int(\p-\f)\om\wedge T\right)^{\frac12}
     +
      \left(\int(\p-\f)(\om+dd^c\f)\wedge T\right)^{\frac12}
    \right)\\
    \le2M^{\frac12}\left(
      \int(\p-\f)(\om+dd^c\frac{\f}{2})\wedge T
    \right)^{\frac12}.
  \end{multline*}
  The lemma follows 
  (with the constant $4M/(2M)^{\frac1{2^n}}<4M$)
  by repeating this argument $n-1$
  times, successively replacing $\f_2,\dots,\f_n$ by $\f/2$.
\end{proof}
\begin{proof}[Proof of Proposition~\ref{P201}]
  Let $(\f_j)_j$ be a decreasing \emph{net} of
  $\om$-psh model functions converging to $\f$.
  We may assume that $-M\le\f_j\le0$ for all $j$,
  where $M\ge 1$.
  For any $\om$-psh function $\p$ with $-1\le \p\le 0$
  it follows from Lemma~\ref{L202} that
  \begin{equation*}
  0\le\int(\f_j-\f)\MA(\p)\le 4M
  \left(\int(\f_j-\f)\MA(\frac\f2)\right)^{\frac1{2^n}}
  \end{equation*}
  and the right hand side tends to zero as $j\to\infty$ by Theorem~\ref{T202}.
  It therefore follows from the definition of the capacity and
  from Chebyshev's inequality that for each integer $m\ge 1$
  there exists $j_m$ such that the open set 
  $G_m:=\{\f_{j_m}-\f>\frac1m\}$ has capacity $2^{-m}\e$.
  We can then set $G:=\bigcup_mG_m$ and $\f_m:=\f_{j_m}$.
\end{proof}
\begin{proof}[Proof of Proposition~\ref{P202}]
  As above let $(\f_j)_j$ be a decreasing \emph{net} of
  $\om$-psh model functions converging to $\f$.
  After adding a constant we 
  may assume that $\f_j\le 0$ for all $j$.
  For each integer $m\ge1$, the net $(\max\{\f_j,-m\})_j$
  decreases to the bounded $\om$-psh function $\max\{\f,-m\}$.
  We can therefore choose $j_m$ such that 
  \begin{equation}\label{e207}
   0\le \int\left(\max\{\f_{j_m},-m\}-\max\{\f,-m\}\right)
    \MA\left(\frac{\max\{\f,-m\}}{2}\right)
    \le(2m)^{-2^{n+1}}.
  \end{equation}
  We may further assume $j_{m+1}\ge j_m$ for all $m$.
  Set $\f_m:=\f_{j_m}$. We claim that the decreasing sequence $(\f_m)_{m=1}^\infty$
  converges  to $\f$.
  By Theorem~\ref{thm:proper} it suffices to test this at any divisorial point 
  $x\in X$. 
  We have $0\ge\f(x)>-\infty$ and 
  $\f_m(x)\ge\f(x)\ge-m$ for $m\ge-\f(x)\ge 0$.
  By~\eqref{e207}, Lemma~\ref{L202} and the definition of capacity we get
  \begin{equation*}
    0\le(\f_m(x)-\f(x))\Capa_\om\{x\}\le\frac1m
  \end{equation*}
  for $m\ge|\f(x)|$. Now $\Capa_\om\{x\}>0$ by Lemma~\ref{L203}, thus $\f_m(x)$ converges to $\f(x)$, which concludes the proof.
\end{proof}
% 
%
%%%%%%%%%%%%%%%%%%%%%%%%%%%%%%%%%%%%%%%%%%%%%%%%%%%%%%%%%%%%%%%%%%%
%
%
\section{Locality and the comparison principle}\label{S104}
Let $\om$ be a form as in~\S\ref{S103} with $\{\om\}^n=1$.
In this section we prove the following analogue of \cite[Proposition 4.2]{BT2}. 
\begin{thm}\label{T101}
  If $\f$ and $\p$ are bounded $\om$-psh functions, then 
  \begin{equation}\label{eq:compar}
  \one_{\{\f>\p\}}\MA(\max\{\f,\p\}) = 
  \one_{\{\f>\p\}}\MA(\f).
  \end{equation}
\end{thm}
A first consequence is the fact that our operator $\MA$ is local in nature, something that
is not an immediate consequence of our definition in~\S\ref{S102}.
\begin{cor}\label{C200}
Suppose $\f$, $\p$ are bounded $\om$-psh functions that agree on an open set
$G\subseteq\Xan$. Then $\MA(\f)=\MA(\p)$ on $G$.
\end{cor}
\begin{proof}
Given $\e>0$ we apply Theorem~\ref{T101} to 
$\f+\e$ and $\p$. This gives $\MA(\max\{\f+\e,\p\})=\MA(\f)$ 
on $G\subseteq\{\f+\e>\p\}$. Letting $\e\to 0$ and using Theorem~\ref{T202} we get
$\MA(\max\{\f,\p\})=\MA(\f)$ on $G$. Exchanging the roles
of $\f$ and $\p$ shows that $\MA(\f)=\MA(\p)$ on $G$.
\end{proof}
Another key consequence of Theorem~\ref{T101} is the \emph{comparison principle}:
\begin{cor}\label{C201}
  If $\f$ and $\p$ are bounded $\om$-psh functions, then
  \begin{equation*}
    \int\limits_{\{\f<\p\}}\MA(\p)
    \le\int\limits_{\{\f<\p\}}\MA(\f).
  \end{equation*}
\end{cor}
\begin{proof} As in \cite[Theorem 1.5]{GZ2}  the result easily follows from the locality property by integration. More precisely, for any $\e>0$ we have
  \begin{multline*}
    1
    =\int\MA(\max\{\f,\p-\e\})
    \ge\int\limits_{\{\f<\p-\e\}}\MA(\max\{\f,\p-\e\})
    +\int\limits_{\{\f>\p-\e\}}\MA(\max\{\f,\p-\e\})\\
    \mathop{=}\limits^{\eqref{eq:compar}}\int\limits_{\{\f<\p-\e\}}\MA(\p-\e)
    +\int\limits_{\{\f>\p-\e\}}\MA(\f)
    =\int\limits_{\{\f<\p-\e\}}\MA(\p)
    +1-\int\limits_{\{\f\le\p-\e\}}\MA(\f),
  \end{multline*}
  so we obtain the desired estimate by letting $\e\to0$.
\end{proof}
The rest of this section is devoted to the proof of Theorem~\ref{T101}.  
We shall use
\begin{lem}\label{L102}
  Let $(\f_j)_j$ be a uniformly bounded net of 
  $\om$-psh functions, and assume that $\MA(\f_j)$ converges to $\MA(\f)$
  in the weak sense of measures for some bounded $\om$-psh function $\f$. Then 
  \begin{equation*}
    \int h\MA(\f_j)\to\int h\MA(\f)
    \quad\text{as $j\to\infty$}
  \end{equation*}
  for every bounded, quasicontinuous function $h$.
\end{lem}
\begin{proof}
  We may assume $0\le h\le 1$, $-M\le\f\le 0$ and 
  $-M\le \f_j\le 0$ for all $j$, where $M\ge 1$.
  Given $\e>0$, let $G$ be an open set such that 
  $\Capa_\om(G)<\e$ and $h$ is continuous on $G^c$, see 
  Definition~\ref{defi:quasicont}.
  Using the Tietze extension theorem, we extend $h|_{G^c}$
  to a continuous function $\tilde{h}$ on all of $\Xan$ such that
  $0\le\tilde{h}\le 1$. We then have
  \begin{align*}
    \int h\MA(\f_j)-\int h\MA(\f)
    &=\int\tilde{h}\MA(\f_j)-\int\tilde{h}\MA(\f)\\
    &+\int\limits_G(h-\tilde{h})\MA(\f_j)
    -\int\limits_G(h-\tilde{h})\MA(\f).
  \end{align*}
  It follows from Lemma~\ref{L201} that 
  \begin{equation*}
    \left|\int h\MA(\f_j)-\int h\MA(\f)\right|
    \le\left|\int\tilde{h}\MA(\f_j)-\int\tilde{h}\MA(\f)\right|
    +2 \sup|h -\tilde{h}| M^n \, \Capa_\om(G).
  \end{equation*}
  Since $\tilde{h}$ is continuous, 
  $\int\tilde{h}\MA(\f_j)\to\int\tilde{h}\MA(\f)$ as $j\to\infty$, thus
  \begin{equation*}
  \limsup_j  \left|\int h\MA(\f_j)-\int h\MA(\f)\right|
    \le 4 \e.
  \end{equation*}
  Letting $\e$ tend to zero completes the proof.
\end{proof}
\begin{proof}[Proof of Theorem~\ref{T101}]
  We prove the result for successively more general 
  functions $\f$, $\p$.

  \smallskip
  \noindent\textbf{Step 1}. 
  First assume $\f$, $\p$ are $\om$-psh model functions. 
  
  Pick an SNC model $\cX$ on which $\f$, $\p$
  and $\max\{\f,\p\}$ are determined by vertical divisors $A,B$ and $C$ respectively. 
  These three functions are then affine on any
  face of the dual complex $\D_\cX$.
  Further, $\MA(\f)$ and $\MA(\max\{\f,\p\})$
  are both atomic measures, supported on 
  divisorial points corresponding to irreducible
  components of the special fiber, see \S\ref{S109}. If $E$
  is such a component for which $\f(x_E)>\p(x_E)$, 
  then $\f(x_F)\ge\p(x_F)$ and hence 
  $\max\{\f(x_F),\p(x_F)\}=\f(x_F)$ for all irreducible components
  $F$ of the special fiber intersecting $E_\om$,
  or else $\max\{\f,\p\}$ would not be affine on the 
  face $[x_E,x_F]$ in $\D_\cX$.
  We have thus shown $\ord_F(A)=\ord_F(C)$ for all components $F$ of $\cX_0$ intersecting $E_\om$. If follows that $A|_E=C|_E$ as numerical classes on $E_\om$, 
  and hence $\MA(\max\{\f,\p\})\{x_E\}=\MA(\f)\{x_E\}$ by definition of Monge-Amp\`ere measures of model functions. 
 
     \smallskip
  \noindent\textbf{Step 2}. 
  Now suppose that $\f$ is an $\om$-psh model function but that $\p$ is merely a
  bounded $\om$-psh function. 
  
  We may assume $- M\le \f,\p < 0$, where $M\ge 1$.
  Note that the set $\Omega:=\{\f>\p\}$ is open since $\f$ is continuous and $\p$ is usc.
  It suffices to prove that $\int h\MA(\max\{\f,\p\})=\int h\MA(\f)$
  for all model functions $h$ whose support is contained in
  $\Omega$ and such that $0\le h\le1$.
  
  Fix a small number $\delta>0$. 
  By Proposition~\ref{P201} there exists an open set $G\subseteq\Xan$ 
  and a decreasing sequence $(\p_j)_{j=1}^\infty$ 
  of $\om$-psh model functions on $\Xan$ such that $\Capa_\om(G)<\d$
  and such that $\p_j$ converges uniformly to $\p$
  on $G^c$.
  Pick $\e>0$ small and rational and write $\Omega_j:=\{\f +\e>\p_j\}$.
  For $j\gg0$, we have $\Omega\cap G^c\subseteq\Omega_j$.
  Since $\f + \e$ and $\p_j$ are both model functions, we have 
  $\MA(\max\{\f + \e ,\p_j\})=\MA(\f)$ on $\Omega_j$ by Step 1.
  It follows from Lemma~\ref{L201} that 
  \begin{align*}
    \left|\int h\MA(\max\{\f +\e ,\p_j\})-\int h\MA(\f)\right|
    &\le\left|\int\limits_Gh\MA(\max\{\f + \e,\p\})
    -\int\limits_Gh\MA(\f)\right|\\
    &\le2M^n\d,
  \end{align*}
  where we have used $0\le h\le1$ and $-M\le\f+ \e ,\p\le0$.
  
  Since $h$ is a model function, it is the difference of two $\om$-psh model 
  functions by Proposition~\ref{Pgenerators}.  
  Now $\max\{\f +\e ,\p_j\}$ decreases to $\max \{ \f, \p\}$ as $j\to\infty$ 
  and $\e \to 0$, so Theorem~\ref{T202} and the above inequality imply
  \begin{equation*}
    \left|\int h\MA(\max\{\f,\p\})-\int h\MA(\f)\right|\le2M^n\d.
  \end{equation*}
  We obtain the desired equality letting $\delta \to0$.
  
  \smallskip
  \noindent\textbf{Step 3}. 
  Finally we treat the general case when $\f$ and $\p$ are bounded $\om$-psh functions. 
  
  Let $(\f_j)_1^\infty$ be a decreasing net of $\om$-psh model functions
  converging to $\f$.
  Write $\Omega_j:=\{\f_j>\p\}$. This is an open set.
  Set $u:=\max\{0,\f-\p\}$. 
  Then
  \begin{equation*}
  \{\f>\p\}=\{u>0\}\subseteq\bigcap_j\Omega_j.
  \end{equation*}
  By what precedes, $\MA(\max\{\f_j,\p\})=\MA(\f_j)$ on $\Omega_j$.
  Moreover $\max\{\f_j,\p\}$ decreases to $\max\{\f,\p\}$ and so
  the measure $\MA(\max\{\f_j,\p\})$ converges weakly to $\MA(\max\{\f,\p\})$.
  Let $f$ be a continuous function on $\Xan$. By Proposition~\ref{P201} $\f, \p$ are quasicontinuous. It follows that $u$ and $fu$  are also
  quasicontinuous,
  and applying Lemma~\ref{L102} twice we get that
  \begin{equation*}
    \int fu\MA(\max\{\f,\p\})
    =\lim_{j\to\infty}\int fu\MA(\max\{\f_j,\p\})
    =\lim_{j\to\infty}\int fu\MA(\f_j)
    =\int fu\MA(\f).
  \end{equation*}
  This holds for every $f\in C^0(\Xan)$,
  so $\one_{\{\f>\p\}}\MA(\max\{\f,\p\})=\one_{\{\f>\p\}}\MA(\f)$,
  as was to be shown.
\end{proof}
% 
%
%%%%%%%%%%%%%%%%%%%%%%%%%%%%%%%%%%%%%%%%%%%%%%%%%%%%%%%%%%%%%%%%%%%
%
%
\section{Energy}\label{S105}
Let $\om$ be a form as in~\S\ref{S103} with $\{\om\}^n=1$. As in the complex case, it turns out that the non-Archimedean Monge-Amp\`ere operator admits a primitive, \ie a functional whose directional derivatives at a given $\f$ are given by integration against $\MA(\f)$. Adapting \cite{GZ2,BEGZ} to our case we introduce and study this functional, as well as the resulting class of $\om$-psh functions of finite energy. 
While such functions are unbounded in general, they
behave from many points of view like bounded $\om$-psh functions.
%
%%%%%%%%%%%%%%%%%%%%%%%%%%%%%%%%%%%%%%%%%%%%%%%%%%%%%%%%%%%%%%%%%%%
%
\subsection{Energy of model functions}
For any model function $\f$ we set 
\begin{equation}\label{e203}
  E_\om(\f)=\frac1{n+1}\sum_{j=0}^n\int\f(\om+dd^c\f)^j\wedge\om^{n-j}
\end{equation}
and call $E_\om(\f)$ the \emph{energy} of $\f$. 
It follows formally from an integration by parts argument, 
see Proposition~\ref{P303} and~\cite[Lemma 6.2]{TianBook} that if
$\f,\p$ are any two model functions, then 
\begin{equation}\label{e202}
  E_\om(\p)-E_\om(\f)
  =\frac1{n+1}\sum_{j=0}^n\int(\p-\f)(\om+dd^c\f)^j\wedge(\om+dd^c\p)^{n-j}.
\end{equation}
Writing $\f_t = (1-t) \f + t \p$, and expanding $E_\om(\f_t) - E_\om(\f)$ in $t$ 
leads to the following formulas for first and second derivatives of $E_\om$:
\begin{align}
  \label{e2008} E'_\om(\f)\cdot(\p-\f) 
  &= \frac{d}{dt}\bigg|_{t=0+} E_\om(\f_t)    =  \int(\p-\f)\MA(\f);\\
  \label{e2009}  E''_\om(\f)\cdot(\p-\f) 
  &= \frac{d^2}{dt^2}\bigg|_{t=0+} E_\om(\f_t)  = n\, \int(\p-\f)dd^c(\p-\f)\MA(\f).
\end{align}
\begin{prop}\label{P304}
  The restriction of $E_\om$ to the convex set 
  $\PSH(X,\om)\cap \cD(\Xan)$
  is concave, nondecreasing, and satisfies  $E_\om(\f+c)=E_\om(\f)+c$ for any constant $c\in\R$.
\end{prop}
\begin{proof}
  Concavity follows from~\eqref{e2009} and Proposition~\ref{prop:cauchys}.
  Monotonicity is a consequence of~\eqref{e2008}, 
  and the last equation follows from~\eqref{e202} since 
    $(\om+dd^c\f)^j\wedge\om^{n-j}$  is a probability measure  for each $j$ thanks to Proposition \ref{P210} and the normalization $\{\om\}^n=1$. 
\end{proof}

%
%%%%%%%%%%%%%%%%%%%%%%%%%%%%%%%%%%%%%%%%%%%%%%%%%%%%%%%%%%%%%%%%%%%
%
\subsection{Energy of $\om$-psh functions}
For a general $\om$-psh function $\f$ we set
\begin{equation*}
  E_\om(\f):=\inf\left\{E_\om(\p)\mid \p\in\PSH(X,\om)\cap \cD(\Xan), \p\ge\f\right\}\in[-\infty,+\infty[.
\end{equation*}
\begin{prop}\label{prop:basicE}
  The extension 
  $E_\om:\PSH(X,\om)\to[-\infty,+\infty[\,$ 
  is non-decreasing, concave, and satisfies $E_\om(\f+c)=E_\om(\f)+c$ for any $c\in\R$. It is also upper semicontinuous, and continuous along decreasing nets 
  \end{prop}
\begin{proof}
  That $E_\om$ is nondecreasing, concave and satisfies $E_\om(\f+c)=E_\om(\f)+c$
  follows formally from Proposition~\ref{P304} (using that $\PSH(X,\om)\cap\cD(X)$ is convex and invariant under addition of a constant). 

 Upper semicontinuity is also a direct consequence of these algebraic properties of $E_\om$ and of Theorem \ref{thm:proper}. Indeed, pick $\f_0\in\PSH(X,\om)$ and $t\in\R$
  such that $E_\om(\f_0)<t$. We need to show that 
  $E_\om(\f)<t$ for $\f$ in a neighborhood $U$ of $\f_0$
  in $\PSH(X,\om)$. By definition, there exists
  $\p_0\in\PSH(X,\om)\cap \cD(\Xan)$ such that 
  $\p_0\ge\f_0$ and $E_\om(\p_0)<t-\e$ for some $\e>0$. By Theorem \ref{thm:proper}, $U:=\{\f\in\PSH(X,\om)\mid \sup_X(\f-\p_0)<\e\}$
  is an open neighborhood of $\f_0$ in $\PSH(X,\om)$. By (\ref{e202}) we have $E_\om(\f)\le E_\om(\p_0)+\e<t$ for all $\f\in U$, which proves upper semicontinuity. 
   
  Finally, being usc and nondecreasing, $E_\om$ is automatically 
  continuous along decreasing nets. 
\end{proof}

\begin{prop}\label{Pforbdd}
Formulas~\eqref{e203}-\eqref{e2009} 
are valid for bounded $\om$-psh functions.
\end{prop}
This follows from the continuity of $E_\om$ along decreasing nets
and from Theorem~\ref{T202}. 
% 
%%%%%%%%%%%%%%%%%%%%%%%%%%%%%%%%%%%%%%%%%%%%%%%%%%%%%%%%%%%%%%%%%%%
%
\subsection{Non-pluripolar Monge-Amp\`ere measures}
Let us introduce the class of $\om$-psh functions with finite energy
\begin{equation*}
  \cE^1(X,\om):=\left\{\f\in\PSH(X,\om)\mid E_\om(\f)>-\infty\right\}.
\end{equation*}
This is a convex set which contains all bounded $\om$-psh functions.

In this section and its sequel, we explain how to extend the Monge-Amp\`ere operator to $\cE^1(X,\om)$ and prove that its basic properties continue to hold 
in this more general setting.

\smallskip
Consider an arbitrary $\om$-psh function $\f$.
In the sequel we shall use the notation
\begin{equation*}
  \ft:=\max\{\f,-t\}.
\end{equation*}
Note that for $s>t\ge 1$, $\{\f>-t\}=\{\fs>-t\}$
and $\max\{\fs,-t\}=\ft$; 
hence Theorem~\ref{T101} implies
\begin{equation*}
  \one_{\{\f>-t\}}\MA(\fs)
  =\one_{\{\fs>-t\}}\MA(\fs)
  =\one_{\{\fs>-t\}}\MA(\ft)
  =\one_{\{\f>-t\}}\MA(\ft).
\end{equation*}

This equation allows us to introduce
\begin{defi}~\cite{BT2,GZ2}
  The \emph{non-pluripolar Monge-Amp\`ere measure} 
  $\MA(\f)$ of any $\om$-psh function $\f$
  is the increasing limit of the measures
  $\one_{\{\f>-t\}}\MA(\ft)$ as $t\to\infty$.
\end{defi}
Here the limit exists in a very strong sense: we have
\begin{equation}\label{e401}
 \lim_{t\to\infty} \one_{\{\f>-t\}}\MA(\ft) (E)  = \MA(\f) (E)
\end{equation}
for any Borel set $E$.
\begin{rmk}\label{rmk:npp}
 The terminology "non-pluripolar" comes from the fact that $\MA(\f)$ does not put mass on pluripolar sets. This in turn follows from Proposition \ref{prop:pshint} applied to the bounded $\om$-psh function $\ft$ and from~\eqref{e401}.
\end{rmk}
The measure $\MA(\f)$ is always defined and 
supported on the set $\{\f>-\infty\}$, but its
total mass may be strictly less than one. 

\begin{defi}
A $\om$-psh function $\f$
has \emph{full Monge-Amp\`ere mass} when $\MA(\f)$ is a probability measure.
\end{defi}

This is the case iff 
$\MA(\ft)\{\f\le-t\}\to0$ as $t\to\infty$,
and implies that $\MA(\ft)$ converges weakly to
$\MA(\f)$.
\begin{lem}\label{L118}
  If $\f\in\cE^1(X,\om)$, then $\MA(\ft)\{\f\le -t\}=o(t^{-1})$ as
  $t\to\infty$; hence $\f$ has full Monge-Amp\`ere mass.
\end{lem}
\begin{proof}
  We may assume $\f\le0$. Set $\mu_t:=\MA(\ft)$. Since~\eqref{e202} applies to bounded $\om$-psh functions by Proposition~\ref{Pforbdd}, we get 
  \begin{multline*}
    E_\om(\ftt)-E_\om(\ft)
    \ge\frac1{n+1}\int(\ftt-\ft)\mu_t
    =\frac1{n+1}\int_0^{t/2}\mu_t\left\{\ftt-\ft\ge s\right\}\,ds\\
    \ge\frac1{n+1}\int_0^{t/2}\mu_t\left\{\ftt-\ft\ge t/2\right\}\,ds
    =\frac{t}{2(n+1)}\mu_t\left\{\f\le-t\right\},
  \end{multline*}
  where $\mu_t=\MA(\ft)$. Since
  $ \lim_{t\to\infty} E_\om(\ftt)= \lim_{t\to\infty} E_\om(\ft) = E_\om(\f)$ by the continuity of $E_\om$  along decreasing sequences,   the proof  is complete.
\end{proof}

\begin{lem}\label{L205}
  If $0\ge\f\in\cE^1(X,\om)$  and $f\in \cD(\Xan)$,
  then 
  \begin{equation*}
    \left|\int f\MA(\ft)-\int f\MA(\f)\right|\le\frac{2(n+1)}{t}|E_\om(\f)|\sup_X|f|
 \end{equation*}
  for any $t>0$.
\end{lem}
\begin{proof}
  We may assume $\sup_X|f|=1$. 
  Pick $s\ge t$. The probability measures 
  $\mu_t:=\MA(\ft)$ and $\mu_s$ agree on $\{\f>-t\}$.
  Hence
  \begin{multline*}
    \left|\int f \mu_t-\int f \mu_s\right|
    \le(\mu_t+\mu_s)\{\f\le-t\}
    \le\frac1t \left(\int -\ft  \mu_t + \int -\fs \mu_s\right)\\
    \le \frac{n+1}t(|E_\om(\ft)|+|E_\om(\fs)|)
    \le\frac{2(n+1)}t|E_\om(\f)|.
  \end{multline*}
  The result follows by letting $s\to\infty$.
\end{proof}

\begin{prop}\label{P203}
  If $\f\in\cE^1(X,\om)$ and $(\f_j)_j$ is a decreasing net
  of $\om$-psh functions converging to $\f$, 
  then $\f_j\in\cE^1(X,\om)$ for all $j$ and
  $\MA(\f_j)\to\MA(\f)$ as $j\to\infty$
  in the weak sense of measures.
\end{prop}
\begin{proof}
  Given $f\in \cD(\Xan)$, we have by definition that 
  $\int f \MA(\ft) \to \int f \MA(\f)$ as $t\to\infty$
  and $\int f \MA(\ftj) \to \int f \MA(\f_j)$
 as $t\to\infty$
  for every $j$. Moreover, Lemma~\ref{L205} shows
  that the latter convergence is uniform in $j$.
  Since for each $t$ we have
  $\int f \MA(\ftj) \to \int f \MA(\ft)$
  as $j\to\infty$ by Theorem~\ref{T202},
  the result follows.
\end{proof}
\begin{lem}\label{L204}
  If $\f,\p\in\cE^1(X,\om)$ and $\f,\p\le0$, then we have the estimate
  \begin{equation*}
    -\infty<E\left(\frac{\f+\p}{2}\right)
    \le\frac{2^{-(n+1)}}{n+1}\int(-\p)\MA(\f).
  \end{equation*}
 \end{lem}
\begin{proof}
  Pick $s,t>0$. Since~\eqref{e203} holds for bounded
  $\om$-psh functions, we see using~\eqref{eqMAmiddle} that
  \begin{equation*}
    -\infty
    <E\left(\frac{\f+\p}{2}\right)
    \le E\left(\frac{\ft+\ps}{2}\right)
    \le\frac{2^{-(n+1)}}{n+1}\int\ps\MA(\ft).
  \end{equation*}
Since $\ps$ decreases to $\p$  at any point of $\Xan$, the right hand side converges to
  \begin{equation*}
    \frac{2^{-(n+1)}}{n+1}\int\p\MA(\ft) 
    \le\frac{2^{-(n+1)}}{n+1}\int_{\{\f>-t\}}\p\MA(\ft)
    =\frac{2^{-(n+1)}}{n+1}\int_{\{\f>-t\}}\p\MA(\f)
  \end{equation*}
  by monotone convergence. We obtain the desired
  estimate by letting $t\to\infty$.
\end{proof}
%
%%%%%%%%%%%%%%%%%%%%%%%%%%%%%%%%%%%%%%%%%%%%%%%%%%%%%%%%%%%%%%%%%%%
%
\subsection{Locality and the comparison principle}
\begin{prop}
For any $\f,\p\in\cE^1(X,\om)$, we have
  \begin{equation}\label{e201}
    \one_{\{\f>\p\}}\MA(\max\{\f,\p\}) = 
    \one_{\{\f>\p\}}\MA(\f),
  \end{equation} and the comparison principle holds: 
\begin{equation}
\label{e:compar-E1}
\int_{\{\f < \p\}} \MA(\p) \le \int_{\{\f < \p\}} \MA(\f).
\end{equation}
\end{prop}
\begin{proof}
  To prove~\eqref{e201}, first assume $\p=-t$, where $t\ge 1$.
  Pick $s\ge t$ so that
  $\ft=\max\{\f_s,-t\}$,
  $\{\f>-t\}=\{\f_s>-t\}$, and 
  \begin{equation*}
    \one_{\{\f>-t\}}\MA(\ft)
    =\one_{\{\f_s>-t\}}\MA(\ft)
    =\one_{\{\f_s>-t\}}\MA(\f_s)
    =\one_{\{\f>-t\}}\cdot\one_{\{\f>-s\}}\MA(\f_s),
  \end{equation*}
  where the second equality follows from Theorem~\ref{T101}.
  As $s\to\infty$, $\one_{\{\f>-s\}}\MA(\f_s)(E) \to \MA(\f)(E)$
  for any Borel set $E$,  so the right hand side of the equation
  above converges to $\one_{\{\f>-t\}}\MA(\f)$.

\smallskip

  Now consider $\f,\p\in\cE^1(X,\om)$ and set $u=\max\{\f,\p\}\in\cE^1(X,\om)$.
  Then 
  \begin{itemize}
\item  $\one_{\{\ft>\pt\}}\MA(u) =\one_{\{\ft>\pt\}}\MA(u^{\langle t\rangle})$
since $\{\ft>\pt\} \subseteq \{u>-t\}$;
\item
$\one_{\{\ft>\pt\}}\MA(u^{\langle t\rangle}) =\one_{\{\ft>\pt\}}\MA(\ft)$ by~\eqref{eq:compar} applied
to $\ft$ and $\p$,  noticing the inclusion $\{\ft>\pt\} \subseteq \{ \ft > \p\}$;
\item 
 $\one_{\{\ft>\pt\}}\MA(\ft) =\one_{\{\ft>\pt\}}\MA(\f)$ by the previous step and the inclusion
 $\{ \ft > \pt \}\subseteq \{ \ft > -t \}$.
 \end{itemize}
To summarize, we get 
 \begin{equation}\label{e123}
   \one_{\{\ft>\pt\}}\MA(u)
    =\one_{\{\ft>\pt\}}\MA(\f).
  \end{equation}
  Now 
  \begin{align*}
    \one_{\{\ft>\pt\}}\MA(u)
    & =\one_{\{\f>-t\ge \p\}}\MA(u) + \one_{\{\f>\p>-t\}}\MA(u)
  \end{align*}
  As $t\to\infty$ the first term tends to $0$ since $\MA(u)$ puts no mass on the pluripolar set $\{ \p = - \infty\}$ (see Remark~\ref{rmk:npp}), and the second term converges to $\one_{\{\f>\p>-\infty\}}\MA(u) = \one_{\{\f>\p\}}\MA(u)$.
 
  Thus the left-hand side of~\eqref{e123} tends to 
  $\one_{\{\f>\p\}}\MA(u)$
  as $t\to\infty$. Similarly, the right-hand side 
  tends to $\one_{\{\f>\p\}}\MA(\f)$.
  Finally the comparison principle follows exactly as in the proof of Corollary \ref{C201}. The proof is complete.
\end{proof}
%
%%%%%%%%%%%%%%%%%%%%%%%%%%%%%%%%%%%%%%%%%%%%%%%%%%%%%%%%%%%%%%%%%%%
%
\subsection{Differentiability}
\begin{prop}\label{Pdiff}
For any $\f, \p \in \cE^1(X,\om)$, the function
$t \mapsto E_\om( (1-t) \f + t \p)$ is differentiable on $[0,1]$, and 
we have
 \begin{equation}\label{e209}
    E'_\om(\f)\cdot(\p-\f):= \frac{d}{dt}\bigg|_{t=0+} E_\om( (1-t) \f + t \p)=\int(\p-\f)\MA(\f)
  \end{equation}
  for any $\f,\p\in\cE^1(X,\om)$. 
\end{prop}
\begin{proof}
Set  $ h(t):=h_{\f,\p}(t):=E_\om((1-t)\f+t\p)$ for $0\le t\le 1$.
Note that $h$ 
is a polynomial of degree at most $n$ when $\f$ and $\p$ are model functions. 
By continuity of the energy along decreasing nets, the same is true in general.
In particular, $h$ is differentiable on $[0,1]$.

Pick any decreasing sequence $(\p_j)_{j=1}^\infty$ of 
$\om$-psh model functions converging to $\p$. 
Note that $h_{\fs,\p_j} \to h_{\f,\p}$ as polynomials 
when $s\to\infty$ and $j\to\infty$, hence
$\frac{d}{dt}\big|_{t=0+} h_{\fs,\p_j} \to h'(0+)$. 
Since~\eqref{e209} holds true for bounded functions 
by Proposition~\ref{Pforbdd}, it suffices to show
\begin{equation*}
  \lim_{j\to\infty}\lim_{s\to\infty} \int(\p_j-\fs)\MA(\fs) = \int(\p-\f)\MA(\f).
\end{equation*}
First, we have
\begin{align*}
  \int \fs\, \MA(\fs) 
  &=\int_{\{ \f \le -s \}} (-s) \, \MA(\fs)  + \int_{\{ \f > -s \}} \f\MA(\fs)\\ 
  &=\int_{\{ \f \le -s \}} (-s) \, \MA(\fs)  + \int_{\{ \f > -s \}} \f \MA(\f)
\end{align*}
by~\eqref{e201}. By Lemma~\ref{L118} the first term 
of the right hand side tends to $0$, and the second term 
converges to $\int\f\MA(\f)$ since $\MA(\f)$ puts no mass on $\{\f=-\infty\}$.

Second, for fixed $j$ we have 
$\lim_{s\to\infty}\int\p_j\MA(\fs)=\int\p_j\MA(\f)$
since $\p_j$ is continuous.

Finally, Lemma~\ref{L301} yields 
$\lim_{j\to\infty}\int\p_j\MA(\f)=\int\p\MA(\f)$,
completing the proof.
\end{proof}
% 
%
%%%%%%%%%%%%%%%%%%%%%%%%%%%%%%%%%%%%%%%%%%%%%%%%%%%%%%%%%%%%%%%%%%%
%
%
\section{Envelopes and differentiability}\label{S106}
Let $\om$ be a form as in~\S\ref{S103} with $\{\om\}^n=1$.  
As explained in the introduction,  the differentiability of the energy 
is not \emph{a priori}  sufficient to make the variational approach work, \ie to infer that a maximizer of the relevant functional over $\cE^1(X,\om)$ is necessarily a critical point. In order to circumvent this difficulty, we show as in \cite{BB} the differentiability of $E_\om\circ P_\om$, where $P_\om$ is the $\om$-psh envelope operator of \S\ref{sec:envelopes}. This idea was originally introduced by Alexandrov \cite{Alexandrov} in the context of real Monge-Amp\`ere equations.

\begin{defi}\label{D301}
  We say that $\om$ has the \emph{orthogonality property} if   
  \begin{equation*}
    \int(f-P_\om(f))\MA(P_\om(f))=0
  \end{equation*}
  holds for every $f\in C^0(X)$. 
\end{defi}
Since $P_\om(f)\le f$, this property means that $\MA(P_\om(f))$ is concentrated on the contact locus $\{P_\om(f)=f\}$. We refer to Appendix~\ref{A1} for more information on the orthogonality property. 

We can state the main result of this section.
\begin{thm}\label{T105}
  Assume that $\om$ has the orthogonality property.
  Then the composition $E_\om\circ P_\om:C^0(X)\to\R$ 
  is G{\^a}teaux differentiable, with directional derivatives given by
    \begin{equation*}
    \frac{d}{dt}\bigg|_{t=0} E_\om\circ P_\om (f + t g) = \int g\, \MA ( P_\om(f)).
  \end{equation*}
\end{thm}
Before giving a proof of this crucial result, we state and prove a corollary of it 
that we shall need when solving the Monge-Amp\`ere equation.

If $\f\in\PSH(X,\om)$  and $f\in C^0(X)$, observe that $P_\om(\f + f)$ is $\om$-psh (\ie is not identically $-\infty$) since $\f+f$ dominates the $\om$-psh function $\f+\inf_X f$. Furthermore we have $P_\om(\f+f)\le\f+f$ since the latter is usc. 

\begin{cor}\label{C112}
  Assume that $\om$ has the orthogonality property.
  Let $\f\in\cE^1(X,\om)$ and $f\in C^0(X)$.
  Then $P_\om(\f + tf)\in\cE^1(X,\om)$ for all $t\in\R$ and
  \begin{equation*}
    \frac{d}{dt}\bigg|_{t=0} E_\om\circ P_\om (\f + t f) = \int f\, \MA(\f).
  \end{equation*}

\end{cor}
\begin{proof}

Note that $\f+tf\ge\f-|t|\sup_X|f|$ implies $P_\om(\f+tf)\ge\f-|t|\sup_X|f|$, hence $P_\om(\f+tf)\in\cE^1(X,\om)$ for all $t$. We are going to show that

\begin{equation}\label{e126}
    E_\om\circ P_\om  (\f + t f )
    =E_\om(\f)+\int_0^t\left(\int f\MA ( P_\om(\f + sf))\right)ds.
  \end{equation}
For all $t\in\R$. If $\f$ is continuous, then the result follows
  immediately from Theorem~\ref{T105}.

  In general, let $(\f_m)_{m=1}^\infty$ 
  be a decreasing sequence of $\om$-psh model functions 
  converging to $\f$, see Proposition~\ref{P202}.

  For each $t\in\R$ the sequence $(P_\om(\f_m+tf))_{m=1}^\infty$ is a decreasing
  sequence of $\om$-psh functions, and we claim that $\lim_m P_\om(\f_m+tf)=P_\om(\f+t f)$.
  Indeed let $\tilde{\f}_{t} = \lim_m P_\om(\f_m+tf)$.  Since $\f_m + tf \ge \f +t f$, we have $P_\om(\f_m+tf)\ge P_\om(\f+t f)$, hence $\tilde{\f}_{t} \ge P_\om(\f + t f)$.
  Conversely, $ \tilde{\f}_t \le P_\om(\f_m+tf)\le \f_m + tf$ for all $m$, hence $\tilde{\f}_t \le \f + tf$ and it follows
  $\tilde{\f}_t = P_\om(\f + tf)$ as required.

  We apply~\eqref{e126} to $\f_m$:
  \begin{equation}\label{e127}
    E_\om(P_\om(\f_m+tf))
    =E_\om(\f_m)
    +\int_0^t\left(\int f\MA(P_\om(\f_m+sf))\right)ds.
  \end{equation}
  As $m\to\infty$, $E_\om(P_\om(\f_m+tf))$ and $E_\om(\f_m)$
  decrease to $E_\om\circ P_\om(\f+t f)$ and $E_\om(\f)$, respectively
  and by Proposition~\ref{P203}, 
  $\int f\MA(P_\om(\f_m+sf))$ converges to $\int f\MA( P_\om(\f + sf))$
  for each $s$.  
  Finally~\eqref{e126} follows
  from~\eqref{e127} using dominated convergence in view of 
  the upper bound $|\int f\MA(P_\om(\f_m+sf))|\le\sup_X|f|$ for all $m$ and all $s$. 
\end{proof}

\begin{proof}[Proof of Theorem~\ref{T105}]
  We follow the exposition in~\cite[\S4.3]{BB} very closely.
  Arguing as in Corollary \ref{C112} we may assume that $f,g\in\cD(\Xan)$. Set 
  $\mu:=\MA(P_\om(f))$. 
  We need to prove that 
  \begin{equation}\label{e204}
    \frac{d}{dt}\bigg|_{t=0+}E_\om\circ P_\om(f+tg)=\int g \, \mu.
  \end{equation}
  As a first step, we linearize the problem and prove that 
  \begin{equation}\label{e205}
    \frac{d}{dt}\bigg|_{t=0+}E_\om\circ P_\om(f+tg)
    =\frac{d}{dt}\bigg|_{t=0+}\int P_\om(f+tg)\, \mu.
  \end{equation}
  Denote the left and right hand sides of~\eqref{e205} by
  $a$ and $b$, respectively. Note that the one-sided 
  derivatives exist since both $E_\om$ and $P_\om$ are concave.
  
  Since $E_\om$ is concave on the space of bounded $\om$-psh functions, 
  the function $$[0,1] \ni s \mapsto  h(s) := E_\om\left( s P_\om(f+tg) + (1-s) P_\om(f)\right)$$
  is also concave, hence
  \begin{multline*}
   E_\om( s P_\om(f+tg) + (1-s) P_\om(f))\le \\ E_\om\circ P_\om(f)  + \frac{d}{dt}\bigg|_{t=0} h
   = E_\om\circ P_\om(f)  + s \left( \int (P_\om(f+tg) -  P_\om(f)) \mu \right)
  \end{multline*}
  by Proposition~\ref{Pforbdd}.
  Taking $s=1$ and letting $t\to0$ yields $ a \le b$.

  To prove the reverse inequality, fix $\e>0$.
  Then there exists $\d>0$ such that 
  $$ D := \int_X P_\om(f+\d g)\mu -\int_X P_\om(f)\mu \ge\d(b-\e).$$
  Since $\mu$ is the differential of $E_\om$, there exists
  $\g>0$ such that 
$$    E\left((1-t)P_\om(f)+tP_\om(f+\d g)\right)
    \ge E_\om(P_\om(f))+t (D-\d\e)
    \ge E_\om(P_\om(f))+t\d(b-2\e)
 $$ for $0\le t\le\g$.
  The concavity of $P$ yields
  $P_\om(f+t\d g)\ge(1-t)P_\om(f)+tP_\om(f+\d g)$.
  Since $E_\om$ is non-decreasing we get
  $$E_\om\circ P_\om(f+t\d g)
  \ge E\left( (1-t)P_\om(f)+tP_\om(f+\d g) \right)
  \ge
    E ( P_\om(f))+ t\d(b-2\e)$$ for 
  $0\le t\le\g$. Letting $t\to0$ and $\e\to0$ we conclude $a\ge b$. 
  This shows that~\eqref{e205} holds.

\smallskip
  
      In view of~\eqref{e205} it remains to show that 
  \begin{equation}\label{e206}
   \int_X (P_\om(f+tg)-P_\om(f)) \,\mu = t\int_X g\mu +o(t)
  \end{equation}
  as $t\to 0+$. 
  
  Since $P_\om(f) \le f$,   the orthogonality property implies
  $P_\om(f) = f$ for $\mu$-a.e.\ point. We thus have
  $P_\om(f+tg)\le f+tg=P_\om(f)+tg$
  $\mu$-\alme
  We claim that $\mu(\Omega_t) = O(t)$ with
  \begin{equation*}
    \Omega_t:=\{P_\om(f+tg)<P_\om(f)+tg\}.
  \end{equation*}
  Observe that $|P_\om(f+tg) - P_\om(f)  | \le t \sup |g|$ 
  so that the claim implies
  \begin{multline*}
  \int_X (P_\om(f+tg) - P_\om(f) -tg) \mu
   \\ 
  = \int_{\Omega_t} (P_\om(f+tg) - P_\om(f)) \mu +  \int_{\Xan\setminus\Omega_t} (P_\om(f+tg) - P_\om(f)) \mu
\\  =
  \int_{\Omega_t} (P_\om(f+tg) - P_\om(f) -tg) + \int_X tg \, \mu 
\\  \le t \!\int_X\! g \, \mu + \mu(\Omega_t) \sup_{\Xan} |P_\om(f+tg) - P_\om(f) - tg| = t \!\int _X\!g \, \mu + O(t^2)
   \end{multline*}
  which proves~\eqref{e206}.
  
  The estimate of $\mu(\Omega_t)$ is based on the
  comparison principle. 
  Since $g$ is a model function, there exists 
  $C \gg1$, $\p\in \cD(\Xan)$ such that $\p$ and 
  $\p+ g$ are $C\om$-psh by Proposition~\ref{Pgenerators}.
  Note that 
  $\Omega_t=\{P_\om(f+tg)+t\p<P_\om(f)+t(\p+g)\}$, and both functions
  $P_\om(f+tg)+t\p$ and $P_\om(f)+t(\p+g)$ are $(1+Ct)\om$-psh.
  The comparison principle then yields
  \begin{equation*}
    \int_{\Omega_t}\left( (1+Ct) \om + dd^c\left(P_\om(f)+t(\p+g)\right)\right)^n
    \le\int_{\Omega_t} (( 1+Ct) \om + dd^c\left(P_\om(f+tg)+t\p\right))^n
  \end{equation*}
 By  expanding  as polynomials in $t$, we get
 $$( (1+Ct) \om + dd^c\left(P_\om(f)+t(\p+g)\right))^n = 
 (\om+ dd^c P_\om(f) )^n + O(t)
  $$ and 
  $$( (1+Ct) \om + dd^c\left(P_\om(f+tg)+t\p\right))^n = (\om + dd^c P_\om(f+tg))^n + O(t).$$
 From these three estimates we conclude
  \begin{equation*}
    \mu(\Omega_t)
    =\int_{\Omega_t}\MA(P_\om(f))
    \le\int_{\Omega_t}\MA(P_\om(f+tg))+O(t).
 \end{equation*}
 But $\Omega_t \subseteq\{P_\om(f+tg)<f+tg\}$, so the 
 orthogonality property implies that the last integral
 vanishes. This concludes the proof.
\end{proof}

\begin{rmk} Observe that the differentiability property of Theorem~\ref{T105} conversely implies the orthogonality property. 
Indeed, pick $f\in C^0(X)$ and set $g:=P_\om(f)-f$. We claim that $\int g\MA(P_\om(f))=0$. It is enough to prove $\int g\MA(P_\om(f))\ge 0$ 
since $g\le 0$. Now the differentiability property yields 
$$
E_\om(P_\om(f+\e g))=E_\om(P_\om(f))+\e\int g\MA(P_\om(f))+o(\e).
$$
But we have 
$$
f+\e g=(1-\e) f+\e P_\om(f)\ge (1-\e) P_\om(f)+\e P_\om(f)=P_\om(f),
$$
hence 
$E_\om\left(P_\om(f+\e g)\right)\ge E_\om(P_\om(f))$ by monotonicity of $E_\om$, and the result follows.
\end{rmk}

% 
%
%%%%%%%%%%%%%%%%%%%%%%%%%%%%%%%%%%%%%%%%%%%%%%%%%%%%%%%%%%%%%%%%%%%
%
%
\section{The Monge-Amp{\`e}re equation}\label{S110}
In this section we prove
\begin{thm}\label{T301}
  Assume that $\om$ is a form as in~\S\ref{S103} with $\{\om\}^n=1$ that satisfies the orthogonality property 
  (see Definition~\ref{D301}). 
  Let $\mu$ be a probability measure on $X$
  supported on the dual complex of some SNC model of $X$.
  Then there exists a unique, continuous $\om$-psh function 
  $\f\in\PSH(\Xan,\om)$ such that $\MA(\f)=\mu$, and $\sup \f =0$.
\end{thm}
Let us explain how to deduce Theorems~A and~A' from the introduction. 
Let $\om$ be any closed semipositive form with $\{ \om\}$ ample, and $\mu$ be a positive Radon measure of mass $\{ \om \}^n$. Set $\tilde{\om} := \om / ( \{ \om \}^n )^{1/n}$,
and $\tilde{\mu} = \mu / \{ \om \}^n$. 
Assume that $X$ is algebraizable. It follows from Appendix A that $\om$ and $\tilde{\om}$  satisfy the orthogonality property. Applying Theorem \ref{T301} to $\tilde{\om}$ and $\tilde{\mu}$ yields a unique $\tilde{\f} \in \PSH(\Xan, \tilde{\om})$ such that 
$\sup \tilde{\f} =0$ and
$(\tilde{\om} + dd^c \tilde{\f})^n = \tilde{\mu}$.
Theorem A' follows since $(\om + dd^c \f)^n = \mu$ with $ \f =  ( \{ \om \}^n )^{1/n}\,  \tilde{\f}$.

Now consider an ample line bundle $L\to X$ endowed with a 
semipositive model metric $\|\cdot\|$. 
The curvature form $\om=c_1(L,\|\cdot\|)$ is semipositive and $\{\om\}^n= c_1(L)^n$ in view of Proposition~\ref{P210}
and~\eqref{eq:ouf}. Given any positive Radon measure $\mu$ of mass $c_1(L)^n$ and supported on the dual complex of some SNC model of $X$, Theorem~A' thus implies the existence of a unique 
continuous $\om$-psh function $\f$ such that $\MA(\f) = \mu$, and $\sup \f=0$. 
This statement implies  Theorem~A since  $c_1(L, \|\cdot \| e^{-\f})^n = \MA(\f)$ by definition.

\medskip

For the rest of this section  is a form as in~\S\ref{S103} normalized by  $\{\om\}^n=1$

\subsection{Uniqueness}

The uniqueness statement in Theorem~\ref{T301} does not require the orthogonality property.
Following~\cite{Blocki} as in~\cite{GZ2,yuanzhang}, one actually proves:
\begin{prop}Let $\om$ be any semipositive closed $(1,1)$ form. 
Suppose $\MA(\f) = \MA(\p)$ for any two functions $\f, \p \in \cE^1(X,\om)$.
Then $\f -\p$ is constant.
\end{prop}

\begin{proof}
For simplicity we write $\om_\f = \om+ dd^c \f$.

First we briefly  indicate how to extend to $\om$-psh of finite energy the
calculus that we developed in \S\ref{S102}. Let $\f,\p\in \cE^1(X,\om)$.
Since $E_\om$ is convex, $(1-t)\f + t\p\in \cE^1(X,\om)$ for any $t \in [0,1]$.
For any $0\le i\le n$, define $\om_\f^i\wedge \om_\p^{n-i}$ 
to be the unique probability measure such that 
\begin{equation}\label{def:finite-energy}
\sum_{k=0}^{n-1}  \binom{n}{k} (1-t)^k t^{n-k} 
\om_{\f}^k\wedge \om_{\p}^{n-k}
=
\MA((1-t) \f + t \p)
-(1-t)^n\, \MA(\f)
- t^n \, \MA(\p)
\end{equation}
for any $t = j/n$ with $1\le j \le n-1$.
By Proposition~\ref{P203}, we get $\om_{\f_j}^i\wedge \om_{\p_j}^{n-i} \to \om_\f^i\wedge \om_\p^{n-i}$  for any decreasing sequence of $\om$-psh functions $\f_j \to \f$ and $\p_j \to \p$. 
In particular, $\om_\f^i\wedge \om_\p^{n-i}$ is a probability measure.
Replacing $\MA(\cdot) = (\om + dd^c \cdot)^n$ by $(\om + dd^c \cdot)^{i+j} \wedge \om^{n- (i+j)}$ in ~\eqref{def:finite-energy}, we can further define probability measures of the same mass
$\om^i_\f\wedge \om^j_\p\wedge \om^{n-(i+j)}$ as soon as $i,j\ge 0$ and $i+j\le n$.

Observe that by definition and Lemma~\ref{L204},  these measures integrate $\om$-psh functions of finite energy. By continuity,  it also follows that the Cauchy-Schwarz inequality holds 
$$
\int - h dd^c g \wedge T  \le
\left(\int h dd^c h \wedge T \right)^{1/2}\, 
\left( \int g dd^c g \wedge T \right)^{1/2}
$$
for any $h, g$ lying in the vector space generated by $\cE^1(X,\om)$ and for any $T$ a positive
linear combination of measures of the type $\om^i_\f\wedge \om^j_\p\wedge \om^{n-(i+j)}$ with $i+j \le n$ and $\f,\p\in \cE^1(X,\om)$.

\medskip

Now pick $\f,\p \in \cE^1(X,\om)$.
We claim that 
\begin{equation}
\int 
(\f-\p) dd^c (\f-\p) \wedge \om^{n-1} 
\le C\, \left( \int (\p-\f) ( \MA(\f) - \MA(\p)) \right)^{2^{1-n}}
\end{equation}
for some constant $C$ depending on $\f$ and $\p$.

Grant this claim, and suppose $\MA(\f)= \MA(\p)$. We conclude the proof as in~\cite{yuanzhang}. We may  assume $\sup \f = \sup \p$. 
By Cauchy-Schwarz inequality, for any model function $h$ we get
$$
\int (\f-\p) dd^c h \wedge \om^{n-1} \le D^{1/2} \left| \int 
(\f-\p) dd^c (\f-\p) \wedge \om^{n-1} \right|^{1/2} =0~,
$$
with $0 \le D:= \int -h dd^c h \wedge \om^{n-1} <+\infty$.
Let $\cL$ be any $\R$-line bundle in a model $\cX$ whose numerical class is equal to $\om$. 
The above equality applied to the model function
determined in $\cX$ by $\sum_E b_E(\f-\p) (\ord_E) E$ with $\cX_0 = \sum b_E E$
yields
$$
\left(\sum_E b_E(\f-\p) (\ord_E) E\right)^2 \cdot \cL^{n-1} =0
$$
which in turn implies $\sum_E b_E(\f-\p) (\ord_E) E$ to be proportional to $\cX_0$ by~\cite[Theorem 2.1.1(b)]{yuanzhang}. Since we normalized $\f,\p$ by $\sup \f = \sup \p$, we conclude that 
$\f = \p$ on the vertices of $\D_\cX$.

Now consider any (sufficiently) high model $\pi: \cX' \to \cX$. By~\cite[Proposition 5.2]{siminag} there exists a model function $h$ such that $\om'=  \om + dd^c h$ is induced by a ample
divisor in $\cX'$. Then the functions $\f-h$ and $\p-h$ are both $\om'$-psh, normalized by $\sup (\f-h) = \sup (\p-h)$,  and
satisfy $(\om' + dd^c (\f-h))^n = (\om'+ dd^c (\p-h))^n$. By what precedes we get $\f =\p$ on the vertices of $\cX'$.
This implies $\f = \p $ on $\Xdiv$, hence on $\Xqm$ by Proposition~\ref{P302}, hence on $X$
since $ \f = \sup_\cX \f \circ p_\cX$ for any $\om$-psh function by~\cite[Proposition 7.6]{siminag}.

\smallskip

We now prove the claim.  For this we reproduce the argument of~\cite{Blocki}.
By $C$ we will denote possibly different constants depending on $\om,\f,\p$. Set $\rho = \f - \p$. For $k = 0, 1, ..., n-1$ we will prove inductively that
$$
0\le \int - \rho dd^c \rho\wedge \om_\f^i \wedge \om_\p^j \wedge \om^k \le
C a^{2^{-k}}$$
where
$$
a=\int (\p-\f) (\MA(\f) - \MA(\p))=
\int -\rho dd^c \rho \wedge T	\ge0
$$
with $T = \sum_{l=0}^{n-1} \om_\f^l \wedge \om_\p^{n-1-l}$,
and $i, j$ are such that $i + j + k = n - 1$. For $k = n - 1$ we will then obtain the desired estimate.

If $k = 0$, then
$$
\int \rho dd^c\rho \wedge \om_\f^i \wedge \om_\p^j
\le \int \rho dd^c \rho \wedge T = a$$
Assume that (3.1) holds for $0,1,...,k-1$. We have 
$$\om_\f^i \wedge \om_\p^j\wedge \om^k= 
\om_\f^{i+k} \wedge \om_\p^j - dd^c \f \wedge \alpha$$ where
$$
\alpha = \om_\f^i \wedge \om_\p^j \wedge \sum_{l=0}^{k-1} \om_\f^l \wedge \om^{k-1-l}
$$
Therefore
\begin{align*}
0 \le \int  - \rho dd^c \rho \wedge \om_\f^i \wedge \om_\p^j \wedge \om^k
&\le
\int
- \rho dd^c \rho \wedge (T- dd^c \f \wedge \a)
\\ &=
- \int \rho dd^c\rho \wedge T
-\int \rho dd^c \f \wedge \alpha \wedge dd^c \rho
\end{align*}
This means that
$$
0\le \int -\rho dd^c \rho \wedge \om_\f^i \wedge \om_\p^j \wedge \om^k
\le a - \int \rho dd^c \f \wedge \alpha \wedge dd^c \rho.
$$
We have
$$
-\int \rho dd^c \f \wedge \alpha \wedge dd^c \rho
\le
\left| 
\int \rho dd^c \f \wedge \alpha \wedge \om_\f
\right|
+
\left| 
\int \rho dd^c \f \wedge \alpha \wedge \om_\p
\right|.
$$
If $\eta$ is equal to $\f$ or $\p$, the Cauchy-Schwarz inequality gives
$$
\left| \int \rho dd^c \f \wedge \alpha \wedge \om_\eta \right|
\le
\left(
\int \rho dd^c \rho \wedge \alpha \wedge \om_\eta
\right)^{1/2}
\,
\left(
\int \f dd^c \f \wedge \alpha \wedge \om_\eta
\right)^{1/2}
$$
By the inductive assumption, we have
$\int -\rho dd^c \rho \wedge \alpha \wedge \om_\eta
\le C a^{2^{-(k-1)}}$, and 
since $\f \in \cE^1(X,\om)$ we get 
$ \int -\f dd^c \f \wedge \alpha \wedge \om_\eta <+ \infty$.
The proof is complete.
\end{proof}
%
%%%%%%%%%%%%%%%%%%%%%%%%%%%%%%%%%%%%%%%%%%%%%%%%%%%%%%%%%%%%%%%%%%%
%
\subsection{Existence}
 As in~\cite{BBGZ}, the strategy  is to first 
  use a variational argument going back to 
  Alexandrov~\cite{Alexandrov} in order to produce
  a solution $\f\in\cE^1(X,\om)$. 

  Consider the functional $F_\mu:\PSH(X,\om)\to[-\infty,+\infty[$ defined by
  \begin{equation}\label{e125}
    F_\mu(\f):=E_\om(\f)-\int \f \mu.
  \end{equation}
  
  We first claim that $F_\mu$ is usc on $\PSH(X,\om)$. 
  By Proposition~\ref{prop:basicE}, $E_\om$ is usc so that 
  it is sufficient to prove $\f \mapsto \int \f \mu$ is continuous on $\PSH(X,\om)$.
  Pick a net $\f_k \to \f$ in $\PSH(X,\om)$, i.e.\ $\f_k(x) \to \f(x)$ for any $x \in \Xdiv$. 
  Since divisorial points are dense in $\D_\cX$ by \cite{jonmus}, and the family $\{ \f_k|_{\D_\cX}\}$
  is equicontinuous by Theorem~\ref{Pconvex}, it follows that $\f_k|_{\D_\cX} \to \f|_{\D_\cX}$
  uniformly. Whence $\int \f_k \mu \to \int \f \mu$ since $\D_\cX$ 
  contains the support of $\mu$ by assumption.
  
  \smallskip
  Now write $\PSH_0(X,\om):=\{\f\in\PSH(X,\om)\mid\sup\f=0\}$, and
  observe that $F_\mu ( \f + c) = F_\mu(\f)$ for any constant $c$ 
  by Proposition~\ref{prop:basicE}, 
  so that $\sup_{\PSH(X,\om)} F_\mu = \sup_{\PSH_0(X,\om)} F_\mu$.
  Since $F_\mu$ is usc, and $\PSH_0(X,\om)$ is compact by 
  Theorem~\ref{thm:proper}, it actually attains its maximum.  
  We can thus find $\f\in\PSH_0(X,\om)$ such that
  $$
  F_\mu(\f) = \sup_{\PSH(X,\om)} F_\mu.
  $$
  Clearly $E_\om(\f)>-\infty$, so $\f\in\cE^1(X,\om)$.
  Let us show that $\MA(\f)=\mu$.
  Pick any model function $f\le 0$ on $\Xan$.
  For $t\in\R$, consider the function
  \begin{equation*}
    h(t)=E_\om\circ P_\om (\f +t f)-\int (\f+ tf) \mu.
  \end{equation*}
  In view of Corollary~\ref{C112}, $h(t)$ is 
  differentiable at $t=0$ with derivative
  $$h'(0)=\int f \,\MA(\f) - \int f \, \mu.$$
  But since $P_\om(\f +t f)\le\f+tf\le 0$, it follows that 
  $h(t)\le F_\mu \circ P_\om(\f +t f) \le F_\mu(\f)=h(0)$
  for all $t$. Thus $h$ has a local maximum at $t=0$,
  so $h'(0)=0$, that is
  $\int f  \mu = \int f \MA(\f)$.
  This implies $\MA(\f)=\mu$, as
  $f$ was an arbitrary model function.
% 
%%%%%%%%%%%%%%%%%%%%%%%%%%%%%%%%%%%%%%%%%%%%%%%%%%%%%%%%%%%%%%%%%%%
%
\subsection{Continuity}
Finally we show that $\f$ is continuous. 
For this we use capacity estimates in the 
spirit of Ko{\l}odziej~\cite{kolodziej1,kolodziej2};
see also~\cite{EGZ}. 
The following result (and its proof) is a translation of~\cite[Lemma 2.3]{EGZ}.
\begin{lem}\label{L124}
  Let $\f,\p\in\cE^1(X,\om)$ with $\p\le 0$. 
  Then
  \begin{equation*}
    \Capa_\om\{\f<\p\}
    \le t^{-n}\int_{\left\{\f<(1-t)\p+t\right\}}\MA(\f)
  \end{equation*}
  for $0<t<1$.
\end{lem}
\begin{proof}
  Fix $u\in\PSH(X,\om)$ with $0\le u\le 1$ and 
  set $\p_t:=(1-t)\p+tu$. 
  We have 
  \begin{equation*}
    \{\f<\p\}
    \subseteq\{\f<\p_t\}
    \subseteq\{\f<(1-t)\p+t\}.
  \end{equation*}
  since $\p\le0$. 
  Now $\MA(\p_t)\ge t^n\MA(u)$ by~\eqref{eqMAmiddle}, so 
  \begin{multline*}
    t^n\int_{\{\f<\p\}}\MA(u)
    \le\int_{\{\f<\p\}}\MA(\p_t)
    \le\int_{\{\f<\p_t\}}\MA(\p_t)\\
    \le\int_{\{\f<\p_t\}}\MA(\f)
    \le\int_{\{\f<(1-t)\p+t\}}\MA(\f),
  \end{multline*}
  where the third inequality follows from the 
  comparison principle~\eqref{e:compar-E1}. Taking the supremum
  over $u$ completes the proof.
\end{proof}
As a consequence, we get the following version of the 'domination principle', sufficient for our purpose. 

\begin{lem}\label{lem:dompple} Let $\f\in\PSH(X,\om)\cap C^0(X)$ and $\p\in\cE^1(X,\om)$. Assume that $\nu:=\MA(\p)$ is supported in the dual complex $\D_\cX$ of some SNC model $\cX$, and that $\f\le \p$ $\nu$-a.e. Then $\f\le \p$ on $X$.
\end{lem}

\begin{proof} Upon translating by a constant we may assume that $0\ge \f\ge -C$. Let $\e>0$. If we choose $0<t\ll 1$ such that $t(C+1)\le\e/2$ then we have 
$$
\nu\{\p+\e<(1-t)\f+t\}\le\nu\{\p+\e/2<\f\}=0. 
$$
By Lemma~\ref{L124} it follows that
$$
\Capa_\om\{\p+\e<\f\}\le t^{-n}\nu\{\p+\e<(1-t)\f+t\}=0
$$
(since $\MA(\p+\e)=\nu$). But $\{\p+\e<\f\}$ is open by continuity of $\f$, hence empty by Lemma~\ref{L203}. We have thus proved that $\f\le \p+\e$ on $X$ for all $\e>0$, and the result follows. 
\end{proof}

Now let $\f\in\cE^1(X,\om)$ be a solution to 
$\MA(\f)=\mu$, with $\mu$ supported in a dual complex $\D_\cX$. 
We may normalize $\f$ by $\sup_X\f=-1$. Let $(\f_j)_j$ be a decreasing net of  
$\om$-psh model functions converging to $\f$. We are going to show that $\f_j\to\f$ uniformly on $X$, which will in particular imply that $\f$ is continuous. 

By Theorem \ref{thm:proper} we have $\sup_X\f_j\to\sup_X\f$, so  we may assume $\f_j\le 0$ for all $j$. Fix $\e>0$. Since $\f$ is continuous on $\D_\cX$, the monotone convergence $\f_j\to\f$ is uniform on $\D_\cX$ by Dini's lemma. We thus have $\f_j\le\f+\e$ $\mu$-a.e. for $j\gg 1$, and Lemma \ref{lem:dompple} yields $\f_j\le\f+\e$ on $X$, which concludes the proof.
%%%%

\subsection{An alternative approach}\label{sec:other-approach}
We now give a more explicit description of the solution to $\MA(\f)=\mu$, when $\mu$ is a finite sum of Dirac masses at divisorial points.
Let $\om$ be a form as in \S\ref{S103} (not necessarily normalized), and assume that $\om$ satisfies the orthogonality property.

\begin{lem}\label{lem:support} Let $S=\{x_1,...,x_N\}\subset\Xdiv$ be a finite set of divisorial points, and set for $t=(t_1,...,t_N)\in\R^N$
\begin{equation}\label{equ:fist}
\f_{S,t}:= \sup\left\{ \f\mid\f\in\PSH(X,\om),\,\f(x_i) \le t_i\text{ for }i=1,...,N\right\}~.
\end{equation}
Then $\f_{S,t}$ is a continuous $\om$-psh function, and $\MA(\f_{S,t})$ is supported in $S$.
\end{lem} 
\begin{proof} Let $\cX$ be an SNC model such that all $x_i$ appear as vertices of $\D_\cX$. By Theorem \ref{thm:proper} there exists a constant $M>0$ such that $\sup_X\f\le M$ for all $\f\in\PSH(X,\om)$ such that $\f(x_1)\le t_1$. Since adding a constant $c$ to the $t_i$ only replaces $\f_{S,t}$ with $\f_{S,t}+c$, we may thus assume $t_i\le -1$ and $\f\le -1$ as soon as $\f\in\PSH(X,\om)$ satisfies $\f(x_1)\le t_1$. Now let $f_\cX\in\cD(X)_\R$ be the unique function that is linear on the faces of $\D_\cX$, takes value $t_i$ at $x_i$ for each $i$, $0$ at any other vertex of $\D_\cX$, and such that $f_\cX=f_\cX\circ\retr_\cX$. Since each $\f\in\PSH(X,\om)$ is convex on the faces of $\D_\cX$ and satisfies $\f\le\f\circ\retr_\cX$, we have $\f(x_i)\le t_i$ for all $i$ iff $\f\le f_\cX$, hence $\f_{S,t}=P_\om(f_\cX)$. This already shows that $\f_{S,t}$ is continuous and $\om$-psh, and the orthogonality property further shows that $\MA(\f_{S,t})$ is supported in $\{f_{S,t}=f_\cX\}$ for each SNC model $\cX$ as above. We thus see that $\supp\MA(\f_{S,t})\subset\bigcap_\cX\{f_\cX< 0\}$. 

We claim that the latter intersection is in fact equal to $\{x_1,...,x_N\}$, which will conclude the proof of the lemma. For each model $\cX$ and each $x\in X$ we may consider the center (or reduction) $c_\cX(x)\in\cX_0$. Let $E_i\in\Div_0(\cX)$ be the component of $\cX_0$ with generic point $c_\cX(x_i)$, and let $\f_{\cX,i}$ be the model function determined by $E_i$. For each $x\in X$ we have $f_\cX(x)=f_\cX(p_\cX(x))=\sum_i t_i\f_{\cX,i}(x)$, hence
$$
\bigcap_\cX\{f_\cX< 0\}=\bigcup_i\bigcap_\cX\left\{x\in X\mid c_\cX(x)\in\overline{c_\cX(x_i)}\right\}=\{x_1,...,x_N\}.
$$
\end{proof}

As a consequence of this result, for any divisorial point $x\in\Xdiv$ then 
\begin{equation}\label{equ:fiex}
\f_x:=\sup\left\{\f\mid\f\in\PSH(X,\om),\,\f(x)\le 0\right\}
\end{equation}
solves $\MA(\f_x) = \{ \om\}^n \, \delta_x$, since the two measures have the same mass. More generally we have:

\begin{prop}\label{prop:envelop} Let $S=\{x_1,...,x_N\}\subset\Xdiv$ be a finite set of divisorial points and let $\mu$ be a positive Radon measure of mass $\{\om\}^n$ with support contained in $\{x_1,...,x_N\}$. Then there exists $t\in\R^N$ such that the function $\f_{S,t}$ defined by (\ref{equ:fist}) solves $\MA(\f_{S,t})=\mu$. 
\end{prop}
\begin{proof} By Theorem A', we can choose $\f$ be a continuous $\om$-psh function satisfying $\MA(\f) =\mu$. Set $t_i=\f(x_i)$ for $i=1,...,N$. We claim that $\f_{S,t}=\f$, which will conclude the proof. On the one hand we have $\f\le\f_{S,t}$ by (\ref{equ:fist}), since $\f$ is $\om$-psh and satisfies $\f(x_i)\le t_i$. On the other hand we have $\f_{S,t}=\f$ on the support of $\MA(\f)$, hence $\f_{S,t}\le \f$ by Lemma \ref{lem:dompple}. 
\end{proof}

\begin{rmk}\label{rmk:KT} Consider the setting of Theorem A, \ie $\{\om\}$ is the class of an (ample) line bundle $L$ on $X$. The strategy proposed in the preliminary work \cite{KoTs} to solve Monge-Amp\`ere equations mostly deals with the case of a Dirac mass $\mu$ at a divisorial point $x\in\Xdiv$. The authors introduce the envelope (\ref{equ:fiex}), and assume by contradiction that $\MA(\f_x)$ is not supported at $x$. They define a limit functional $F$ obtained by looking at the asymptotics of ball volumes in the space of sections of $mL$ as $m\to\infty$, and indicate that $F$ should satisfy $F(\f_x+\e f)=F(\f_x)+\e\int f\MA(\f_x)+O(\e^2)$ for each $f\in C^0(X)$. Comparing with \cite{BB} in the complex case, $F$ is likely to coincide with $E_\om\circ P_\om$, so that a version of the differentiability property
(Theorem~\ref{T105}) would also be a key ingredient in the approach proposed in \cite{KoTs}. 
\end{rmk}

\begin{rmk}\label{rmk:regenv} We do not know whether the function $\f_x$ in~\eqref{equ:fiex} is necessarily a model function. This is the case on a toric variety,  see Proposition \ref{prop:toricdirac} below, but we suspect the answer is no in general.

Pick an SNC model $\cX$, an extension $\cL\in\Pic(\cX)_\Q$ of $L$, let $\om$ be the curvature form of the model metric defined by $\cL$. Let also $E$ be a component of $\cX$ corresponding to the divisorial point $x= x_E$. We have
$\f_x=P_\om(- f_E )$ up to a constant. On the other hand,
by~\cite[Theorem 8.5]{siminag}, 
$$
P_\om(-f_E)=\lim_m\frac1m \log |\fa_m|
$$ 
where $\fa_m$ denotes the base-ideal of 
$m \cL'$ with  $\cL':=\cL-E$. As a consequence, $\f_x$ is indeed a model function as soon as the graded $S$-algebra $\bigoplus_{m\ge 0} H^0 (\cX,m \cL')$ is finitely generated. 
Building on Nakayama's counterexample to the existence of Zariski decompositions \cite{Naka},
it is reasonable to expect this algebra not to be finitely generated in general, and  to subsequently prove that $\f_x$ is not a model function.
\end{rmk} 

%
%%%%%%%%%%%%%%%%%%%%%%%%%%%%%%%%%%%%%%%%%%%%%%%%%%%%%%%%%%%%%%%%%%%
%
%

\section{Curves and toric varieties}\label{sec:examples}

%%%%

\subsection{Curves}
Potential theory on non-Archimedean analytic curves (over arbitrary complete valuation fields) was developed in detail by A.Thuillier in~\cite{thuillierthesis}. We only indicate how to recover Theorem A' when $\dim X=1$ following his approach. 

Let $X$ be a smooth projective curve over $K$. Thuillier defined spaces $D^0(X)$ and $D^1(X)$ of distributions and currents on $X$ as follows. An element of $D^0(X)$ is  an arbitrary function $\Xqm\to\R$ \cite[Proposition 3.3.3]{thuillierthesis}. The $dd^c$-operator extends to $dd^c:D^0(X)\to D^1(X)$, and its image is exactly the set of currents $\rho\in D^1(X)$ such that $\int_X\rho=0$ \cite[Th\'eor\`eme 3.3.13]{thuillierthesis}. By linearity, this fact easily reduces to the existence, for any two $x,y\in\Xdiv$, of a 'Green function', \ie a model function $g_{x,y}$ such that $dd^c g_{x,y}=\delta_x-\delta_y$. The existence of $g_{x,y}$ is in turn a consequence of the intersection form being negative definite on $\Div_0(\cX)_\R/\R\cX_0$, for a model $\cX$ such that $x$ and $y$ correspond to components of $\cX_0$. 

Now let $\om$ be a $(1,1)$-form with $\int\om>0$, and let $\mu$ be an arbitrary positive Radon measure on $X$ such that $\int\mu=\int\om$. The previous result shows the existence of a distribution $\f_\mu$ such that 
\begin{equation}\label{eq:laplace}
\om + dd^c \f_\mu = \mu. 
\end{equation}
By \cite[Lemme 3.4.1]{thuillierthesis} the positivity of the current $\om+dd^c\f_\mu$ shows that $\f_\mu$ uniquely extends to a $\om$-psh function, and we conclude that any positive Radon measure $\mu$ with $\int\mu=\int\om$ satisfies (\ref{eq:laplace}) for some $\f\in\PSH(X,\om)$, unique up to an additive constant.

Finally, assume that $\mu$ is supported on a dual complex $\D_\cX$. In order to see that $\f_\mu\in C^0(X)$, we may assume that $\cX$ is also a determination of $\om$. In this one-dimensional setting, it is easy to check that composing with the retraction $p_\cX:X\to\D_\cX$ preserves $\om$-psh functions, \ie $\f\circ\retr_\cX$ is $\om$-psh for every $\om$-psh function $\f$. Since $\mu$ is supported on $\D_\cX$ we have $\left(\retr_\cX\right)_*\mu=\mu$, hence $\theta+dd^c(\f_\mu\circ\retr_\cX)=\mu$. It follows that  $\f_\mu\circ\retr_\cX=\f_\mu$ by uniqueness up to an additive constant, since the two functions coincide on $\D_\cX$. Now $\f_\mu|_{\D_\cX}$ is continuous, hence the continuity of $\f_\mu$. 

\medskip

Let us now make the connection with the approach we followed in higher dimensions. In dimension $1$,
the energy is equal to $E(\f) = 2 \int \f \om + \int \f dd^c \f$ so that
a $\om$-psh function $\f$ has finite energy iff $\f$ is integrable with respect to the trace measure of $dd^c \f$. 

Now fix a positive Radon measure $\mu$ such that  the solution $\f_\mu$ to~\eqref{eq:laplace}  has finite energy. Then $\f_\mu$ is the unique $\om$-psh function realizing the infimum of the functional $E (\f)- \int \f\mu$, by~\cite[Proposition 3.5.9]{thuillierthesis}. 

Observe that the assumption on $\mu$ is automatically satisfied when $\mu$ is supported in some dual complex $\D_\cX$ whence Thuillier's result gives a stronger version than our result in dimension $1$.

\medskip

We refer to~\cite{thuillierthesis} for more on potential theory
on non-Archimedean curves including the notion of harmonic functions, capacity, and the study of polar sets. See also~\cite{BRbook} for the case of the projective line.

%%%%

\subsection{Toric varieties}
We use \cite{FultonToric,KKMS,BPS} as references. Let $M\simeq\Z^n$ be a free abelian group, $N$ its dual, and let $T=\Spec K[M]$ be the corresponding split $K$-torus. A projective toric $K$-variety $X$ is described by a rational fan subdivision $\Sigma$ of $N_\R$, and there is a natural embedding $j:N_\R\to X^\mathrm{an}$ given by monomial valuations that sends
$ n \in N_\R$ to the norm $\sum a_m m \in K[M]  \mapsto \max\{ |a_m| \exp ( -  \langle m , n \rangle) \}$. In particular, $j(0)=x_G$, the Gauss point of the open $T$-orbit. 

An ample $T$-line bundle $L$ on $X$ defines a rational polytope $\D\subset M_\R$ with normal fan $\Sigma$, such that points of $M\cap\D$ identify with $T$-eigensections of $L$. 

According to \cite{BPS} we have the following description of toric metrics on $L$. The polytope $\D$ is the Newton polytope of the piecewise $\Q$-linear convex function $g_\D=\sup_{m\in\D} m$ on the dual space $N_\R=M_\R^*$, and toric bounded (resp. model) metrics $\|\cdot\|$ on $L$ correspond to bounded (resp. piecewise $\Q$-affine) functions $f$ on $N_\R$ such that $f-g_\D$ is bounded. The metric $\|\cdot\|_f$ attached to a function $f$ is semipositive iff $f$ is convex. 

The real Monge-Amp\`ere measure of any convex function $f$ on $N_\R$ is a well-defined positive Radon measure $\MA_\R(f)$ on $N_\R$ (see e.g. \cite{RT}), while the growth condition $f=g_\D+O(1)$ further guarantees that 
$$
\int_{N_\R}\MA_\R(f)=\vol(\D).
$$ 
If $f$ is a convex function on $N_\R$ with $f=g_\D+O(1)$, and if $\|\cdot\|_f$ is the corresponding continuous semipositive metric on $L$, then \cite[Theorem 5.70]{BPS} relates their Monge-Amp\`ere measures as follows:
\begin{equation}\label{equ:redreal}
c_1(L,\|\cdot\|_f)^n=n!\,j_*\MA_\R(f). 
\end{equation}
Since $g_\D$ is homogeneous, $\MA_\R(g_\D)$ is a Dirac mass at the origin of mass $\vol(\D)$, 
and~\cite[p.111]{FultonToric} implies the corresponding metric $\|\cdot\|_{g_\D}$ on $L$ to satisfy
$$
c_1(L,\|\cdot\|_{g_\D})^n= c_1(L)^n\, \d_{x_G}.
$$
Translating in $N_\R$ we get:
\begin{prop}\label{prop:toricdirac} Let $\mu$ be a Dirac mass on $X$ centered at a toric divisorial point $j(x)\in\Xdiv$, $x\in N_\Q$. Then $c_1(L,\|\cdot\|_x)^n=\mu$, where $\|\cdot\|_x$ is the toric \emph{model} metric attached to the convex piecewise $\Q$-affine function $y\mapsto g_\D(y-x)$. 
\end{prop}

In the case of atomic measures supported at toric divisorial points, we can show:
\begin{prop} Let $(X,L)$ be a polarized toric $K$-variety. Pick $x_1,...,x_N\in N_\Q$ and set $\mu_w:=\sum_i w_i\d_{j(x_i)}$ for each $w\in\R_+^N$. Then for a dense set of $w\in\R_+^N\cap\{\sum_i w_i=\deg L\}$ the semipositive toric metric $\|\cdot\|$ solving 
$$
c_1(L,\|\cdot\|_w)^n=\sum_iw_i\d_{j(x_i)}.
$$
is a model metric. 
\end{prop}
\begin{proof} For each $t\in\R^N$ let $f_t$ be the upper envelope of the family of piecewise $\Q$-affine convex functions $f$ on $N_\R$ such that $f=g_\D+O(1)$ and $f(x_i) \le t_i$ for all $i$, and let $\|\cdot\|_t$ be the corresponding continuous toric semipositive metric. By Proposition~\ref{prop:envelop}, each measure $\mu_w$ with $w\in\R_+^N\cap\{\sum_i w_i=\deg L\}$ is of the form $c_1(L,\|\cdot\|_t)^n$ for some $t\in\R^N$. Now elementary Newton polytope considerations show that $f_t$ is piecewise $\Q$-affine when all $t_i$ are rational, and the result follows by continuity of $t\mapsto c_1(L,\|\cdot\|_t)^n$. 
\end{proof}
\begin{rmk}
Results of this section are likely to extend to the case of an arbitrary non-Archimedean complete non-trivially valued field. We refer to~\cite{BPR,gubler} for a discussion of toric varieties in this context.
\end{rmk}

%
%
%%%%%%%%%%%%%%%%%%%%%%%%%%%%%%%%%%%%%%%%%%%%%%%%%%%%%%%%%%%%%%%%%%%
%
%

\appendix
\section{Orthogonality}\label{A1}
Recall that given $\theta\in\cZ^{1,1}(X)$ with ample de Rham class $\{\theta\}\in N^1(X)$ and $f\in C^0(X)$ we say that $(\theta,f)$ satisfies the orthogonality property if 
\begin{equation}\label{equ:ortho}%\tag{\dag}
\int_X\left(P_\theta(f)-f\right)\left(\theta+dd^c P_\theta(f)\right)^n=0
\end{equation}
holds. It is convenient in what follows not to require that $\theta$ be semipositive, as opposed to the main body of the text. 

\begin{lem}\label{lem:orthomodel}
Fix $\a\in N^1(X)$ any ample class. Then the following assertions are equivalent:
\begin{enumerate}
\item 
For any form $\theta$ such that  $\{\theta\} = \a$ and any continuous function $f$, the pair $(\theta,f)$ satisfies the orthogonality property. 
\item 
For any form $\theta$ such that  $\{\theta\} = \a$ and any model function $f$, the pair $(\theta,f)$ satisfies the orthogonality property. 
\item
For any form $\theta$ such that  $\{\theta\} = \a$, the pair $(\theta,0)$ satisfies the orthogonality property. 
\end{enumerate}
\end{lem}
When any of these properties hold, we simply say that the class $\a$ (or $\theta$) satisfies the orthogonality property. 
\begin{proof} 
It is clear that $(1) \Rightarrow (2) \Rightarrow (3)$. 
The implication $(3) \Rightarrow (2)$ follows from the equality $P_\theta(f) -f= P_{\theta + dd^c f}(0)$. 
It remains to prove $(2) \Rightarrow (1)$.
We may write a given $f\in C^0(X)$ as a uniform limit on $X$ of model functions $f_j$, and $P_\theta(f_j)\to P_\theta(f)$ uniformly on $X$ thanks to the Lipschitz property of $P_\theta$, see Proposition~\ref{prop:basicenv}. By Theorem \ref{T202} we thus have 
$$
\left(\theta+dd^c P_\theta(f_j)\right)^n\to\left(\theta+dd^cP_\theta(f)\right)^n
$$ 
in the weak topology of measures. Since $P_\theta(f_j)-f_j\to P_\theta(f)-f$ uniformly on $X$ and the measures $(\theta+dd^c P_\theta(f_j))^n$ have uniformly bounded (in fact, constant) mass, it follows that 
$$
\int\left(P_\theta(f)-f\right)\left(\theta+dd^c\ P_\theta(f)\right)^n=\lim_j\int\left(P_\theta(f_j)-f_j\right)\left(\theta+dd^c\ P_\theta(f_j)\right)^n=0. 
$$
Let us prove the final assertion. Pick $\theta'\in\cZ^{1,1}(X)$ such that $\{\theta'\}=\{\theta\}$ in $N^1(X)$. By the analogue of the $dd^c$-lemma proved in \cite[Theorem 4.3]{siminag} there exists $g\in\cD(X)$ such that $\theta'=\theta+dd^c g$. Observe that a function $\f$ is $\theta'$-psh iff $\f+g$ is $\theta$-psh. As a consequence we get
$P_{\theta'}(f)-f=P_{\theta}(f+g)-(f+g)$, hence
$$
\int\left(P_{\theta'}(f)-f\right)\left(\theta'+dd^c P_{\theta'}(f)\right)^n=
\int\left(P_{\theta}(f+g)-(f+g)\right)\left(\theta+dd^c P_\theta(f+g)\right)^n
$$
for all $f\in C^0(X)$. 
\end{proof} 

\begin{lem}\label{lem:reductions} The set of classes in $N^1(X)$ satisfying the orthogonality property is a closed subset of the ample cone. 
\end{lem}
\begin{proof} 
Pick any regular model $\cX$. Then the linear map $N^1(\cX/S)\to N^1(X)$ is surjective
hence open.
It is thus enough to prove the following claim: let $\theta_\cX\in N^1(\cX/S)$ have ample image in $N^1(X)$, and assume that $\theta_\cX$ is the limit of a sequence $\theta_{m,\cX}\in N^1(\cX/S)$. If the corresponding forms $\theta_m\in\cZ^{1,1}(X)$ all satisfy the orthogonality property, then so does $\theta$. 

Let  $f\in C^0(X)$. By Proposition \ref{prop:basicenv} we have $P_{\theta_m}(f)\to P_\theta(f)$ uniformly on $X$.  We claim that
$$
(\theta_m+dd^c P_{\theta_m}(f))^n\to(\theta+dd^c P_\theta(f))^n$$ 
with uniformly bounded mass. Since $(P_{\theta_m}(f)-f)\to(P_\theta(f)-f)$ uniformly on $X$, we have as before
$$
\int\left(P_\theta(f)-f\right)\left(\theta+dd^c\ P_\theta(f)\right)^n=\lim_m\int\left(P_{\theta_m}(f)-f\right)\left(\theta+dd^c\ P_{\theta_m}(f)\right)^n=0
$$
which concludes the proof.

To prove the claim, pick any model function $g \in \cD(X)$, and fix  $ \e >0$.
By Corollary~\ref{cor:richberg}, we can find a $\theta$-psh model function $\f$ such that $ \sup |\f - P_\theta(f)| \le \e$. We then have
\begin{multline*}
I_m:= \left| \int g\, (\theta_m+dd^c P_{\theta_m}(f))^n
- \int g \, (\theta+dd^c P_\theta(f))^n \right|
\le \\
\left| \int g\, (\theta_m+dd^c P_{\theta_m}(f))^n
- \int g \, (\theta_m+dd^c \f)^n \right|
+
\left| \int g\, (\theta_m+dd^c \f)^n
- \int g \, (\theta+dd^c \f )^n \right|
+ \\
\left| \int g\, (\theta+dd^c\f)^n
- \int g \, (\theta+dd^c P_\theta(f))^n \right|
\end{multline*}
Using integration by parts, the last term can be bounded as follows.
\begin{multline*}
\left|
\int g\, (\theta+dd^c\f)^n
- \int g \, (\theta+dd^c P_\theta(f))^n
\right|
= \\
\left| \int (\f  - P_\theta(f)) \, dd^c g \wedge \sum_{i=0}^{n-1} 
(\theta+dd^c\f)^i
\wedge (\theta+dd^c P_\theta(f))^{n-i-1} \right|
\le
C \e 
\end{multline*}
where $\om$ is a fixed  form such that $(\om + dd^c g)$ is semipositive, 
and  $C = 2 \{ \om \} \, \{ \theta\}^{n-1}$.
In a similar way, the first term is bounded from above by
$$\left| \int g\, (\theta_m+dd^c P_{\theta_m}(f))^n
- \int g \, (\theta_m+dd^c \f)^n \right|
\le C\, \sup |P_{\theta_m}(f) - \f|  \le 2 C \e~,
$$
for $m$ large enough.
Finally $ g$ and $\f$ being model functions, the second term tends to zero as $m\to\infty$, and we get $\limsup_m I_m \le 3C \e$. We conclude by letting $\e\to0$.
\end{proof}

We next translate the orthogonality property into a more geometric condition. Let $\cL$ be a line bundle on a model $\cX$ and assume that $L:=\cL|_X$ is ample. For each $m\in\N$ let $\fa_m$ be the base-ideal of $m\cL$, \ie the image of the evaluation map 
$$
H^0(\cX,m\cL)\otimes\cO_\cX(-m\cL)\to\cO_\cX.
$$
Note that the ideal sheaf $\fa_m$ is vertical (\ie cosupported on $\cX_0$) for $m\gg 1$, thanks to the ampleness condition on the generic fiber. Let $\rho_m:\cX_m\to\cX$ be the normalized blow-up of $\cX$ along $\fa_m$, so that the base-scheme $F_m$ of $\rho_m^*(m\cL)$ is now a vertical Cartier divisor satisfying
$$
\fa_m \cdot \cO_{\cX_m} = \cO_{\cX_m}(-F_m). 
$$ 
Finally, let $\cM_m:=\rho_m^*(m\cL)-F_m$ be the base-point free part. The resulting decomposition 
\begin{equation}\label{equ:zar}
\rho_m^*\cL=\tfrac 1m\cM_m+\tfrac 1m F_m%\tag{\dag\dag}
\end{equation}
is sometimes called an \emph{approximate Zariski decomposition}. 

\begin{lem}\label{lem:geomortho} With the previous notation let $\theta\in\cZ^{1,1}(X)$ be the curvature form of the model metric induced by $\cL$. Then $(\theta,0)$ satisfies (\ref{equ:ortho}) iff the approximate Zariski decompositions (\ref{equ:zar}) are asymptotically orthogonal, in the sense that
\begin{equation}\label{equ:asymzar}
\lim_{m\to\infty}\left(\tfrac 1m\cM_m\right)^n\cdot\left(\tfrac 1m F_m\right)=0.
\end{equation}
\end{lem}

\begin{proof} For all $m\gg 1$ set $\f_m:=\tfrac1m \log|\fa_m|=\tfrac 1m\f_{F_m}$.  
This is a $\theta$-psh model function, and \cite[Theorem 8.5]{siminag} states that $\f_m\to P_\theta(0)$ uniformly on $X$. Unravelling the definitions, we find
$$
-\left(\tfrac 1m\cM_m\right)^n\cdot\left(\tfrac 1m F_m\right)=\int\f_m\,(\theta+dd^c\f_m)^n. 
$$
By Theorem \ref{T202} the right-hand side converges to $\int P_\theta(0)\left(\theta+dd^c P_\theta(0)\right)^n$, which proves the result.
\end{proof}

Recall that $X$ is said to be \emph{algebraizable}  if there exists a (one-variable) function field $F$ admitting $K$ as a completion and a smooth projective $F$-scheme $Y$ such that $X=Y_K$.

\begin{thm}\label{T201}
Let $X$ be an algebraizable smooth projective $K$-variety. Then all ample classes in $ N^1(X)$ have the orthogonality property. 
\end{thm}
 
\begin{proof}[Proof of Theorem~\ref{T201}] 
Let us fix an ample class $\a \in N^1(X)$.
By Lemma \ref{lem:reductions} we may assume  $\a \in N^1(X)_\Q$. 

Since $X$ is algebraizable, we can find a smooth projective curve $B$ over the residue field $k$ such that $F=k(B)$; a closed point $0\in B$ and  a regular parameter $t\in\cO_{B,0}$ inducing an isomorphism $S\simeq\spec\hcO_{B,0}$; and a smooth projective variety $Y$ over $F$ such that $X = Y_K$.

By Lemma \ref{lem:NS} below we may then choose an ample $\Q$-line bundle $L\in\Pic(Y)_\Q$ mapping to $\a$ in $N^1_\Q(X)$. We can also find a normal, flat and projective $B$-scheme $\fY$ having $Y$ as its generic fiber and such that $L\in\Pic(Y)_\Q$ extends to $\fL\in\Pic(\fY)_\Q$. The latter is therefore ample on the generic fiber of the structure morphism $\pi:\fY\to B$, hence in particular $\pi$-big. 
Since the natural morphism $\cX:=\fY\times_B S\to\fY$ is regular, $\cX$ is normal, as well as flat and projective over $S$, hence a model of $X$ according to our definition. The $\Q$-line bundle $\fL$ induces $\cL\in\Pic(\cX)_\Q$.

The curvature form $\theta\in\cZ^{1,1}(X)$ of the model metric defined by $\cL$ has $\a$ as its de Rham class. Our goal is to show that (\ref{equ:ortho}) holds for each $f\in\cD(X)$. We may in fact assume that $f=0$. Indeed let $\cX'$ be a determination of $f$, which may be taken to dominate $\fY$. The model $\cX'$ is then the blow-up of $\cX$ along a vertical ideal sheaf $\fa$. Since $(t^m)\subset\fa$ for some $m\in\N$, $\fa$ comes from an ideal sheaf on $\fY$, and the blow-up $\fY'$ of $\fY$ along this ideal satisfies $\fY' \times_B S=\cX'$ since blow-ups commute with flat base change. Replacing $\fY$ with $\fY'$, we may thus assume that $\cX$ is a determination of $f$, so that there exists a vertical $\Q$-divisor $E\in\Div_0(\cX)$ such that $f=f_E$. Since $E$ is vertical, it also comes from $\fY$. Replacing $\fL$ with $\fL+E$ reduces us as desired to the case $f=0$.

After perhaps passing to a multiple, we may further assume that $\fL\in\Pic(\fY)$. According to Lemma \ref{lem:geomortho}, we are to show that the approximate Zariski decompositions of $\cL$ are asymptotically orthogonal. 

Denote by $\fa_m$ the base-ideal of $m\cL$ on $\cX$, and let $\fb_m\subset\cO_\fY$ be the relative base-ideal of $m\fL$ on $\fY/B$. By flat base change we have $\fa_m=\fb_m\cdot\cO_{\cX}$. Let  $\rho_m:\fY_m\to\fY$ be the normalized blow-up of $\fY$, and let $G_m$ be the effective Cartier divisor of $\fY_m$ such that 
$\cO_{\fY_m}(-G_m)=\fb_m\cdot\cO_{\fY_m}$. Note that $G_m$ is supported on finitely many fibers over $B$ for $m\gg 1$, since $m\fL$ is $\pi$-ample. Observe also that $G_m$ pulls back to the similarly defined divisor $F_m$ on $\cX_{m}:= \fY_m \times_B S$. Finally set $\fM_m:=\rho_m^*(m\fL)-G_m$, which pulls back to $\cM_m$ on $\cX_{m}$. Once again by flat base change, it is enough to show that
$$
\lim_{m\to \infty}\left(\tfrac 1m \fM_m\right)^n\cdot\left(\tfrac 1m G_m\right)=0.
$$
We are going to prove this by reducing to the absolute case of a big line bundle $\fD$ on $\fY$. By Lemma~\ref{lem:gg} below we may choose an ample line bundle $H\in\Pic(B)$ such that the sheaves 
$$
\cO_B(mH)\otimes \pi_*\cO_\fY(m\fL)
$$ 
are globally generated over $B$ for all $m\gg 1$ sufficiently divisible. Since $\fL$ is $\pi$-big, we may assume (after perhaps replacing $H$ with a large enough multiple) that $\fD:=\fL+\pi^*H$ is a big line bundle on the projective $k$-variety $\fY$. 

The relative base-ideal $\fb_m$ of $m\fL$ coincides with the relative base-ideal of $m\fD$ 
since $m\fL$ and $m\fD$ are $\pi$-linearly equivalent by construction. 
The fact that 
$$
\cO_B(mH)\otimes \pi_*\cO_\fY(m\fL)=\pi_*\cO_\fY(m\fD)
$$
is globally generated therefore shows that $\fb_m$ is also the (absolute) base-ideal of $m\fD$. As a consequence we get that 
$\fP_m:=\rho_m^*(m\fD)-G_m$ is the (absolute) base-point free part of $m\fD$, and we infer from \cite[Theorem 4.1]{BDPP} that
$\left(\tfrac 1m \fP_m\right)^n\cdot\left(\tfrac 1m G_m\right)\to 0$.  But  $\fP_m=\fM_m+(\pi\circ\rho_m)^*(mH)$ implies $\fM_m^n\cdot G_m=\fP_m^n\cdot G_m$ since $G_m$ is supported on finitely many fibers over $B$, and the result follows. 

\end{proof}

\begin{lem}\label{lem:NS} Let $Y$ be smooth proper scheme over a field $F$, and let $K / F$ be an arbitrary field extension. Then the natural morphism 
$N^1(Y)_\Q\to N^1(Y_{K})_\Q$ is an isomorphism preserving ample classes.
\end{lem}
Here we write $N^1(Y)_\Q$ for the $\Q$-vector space defined as the quotient of $\Pic(Y)$ by the subspace spanned by  numerically trivial line bundles, i.e. line bundles of degree $0$ over all curves proper over $F$.
Similarly, the N\'eron-Severi group $\NS(Y)$ is the quotient of $\Pic(Y)$
modulo algebraically trivial line bundles, i.e. $\NS(Y) = \Pic(Y)/\Pic^0(Y)$, so that
$\NS(Y)$ is the group of components of $\Pic(Y)$.

When $F$ is \emph{algebraically closed}, then  we have $N^1(Y)_\Q=\NS(Y)_\Q$ by \cite{Mat}.

\begin{proof} Let $\overline{K}/\overline{F}$ be algebraic closures. The groups of components of the Picard groups $\Pic(Y_{\overline{F}})$ and $\Pic(Y_{\overline{K}})$ are then isomorphic, so we have an isomorphism $N^1(Y_{\overline{F}})_\Q\simeq N^1(Y_{\overline{K}})_\Q$ by the result of \cite{Mat} recalled above (compare \cite[Proposition 3.1]{MP}). This isomorphism is furthermore compatible with ample classes by the Nakai-Moishezon criterion for ampleness. 

It is enough to show the surjectivity of $N^1(Y)_\Q\to N^1\left(Y_{K}\right)_\Q$. Let $\b\in N^1(Y_{K})_\Q$. By the previous result, we find $L\in\Pic(Y_{\overline{F}})_\Q$ mapping to the lift of $\b$ in $N^1(Y_{\overline K})_\Q$. Since $Y_{\overline{F}}$ is in particular reduced, $L$ can be represented by some $\bar{D}\in\Div(Y_{\overline F})_\Q$. The average of the Galois orbit of $\bar{D}$ is then $\mathrm{Gal}(\bar{F}/F)$-invariant, hence descends to $D\in\Div(Y)_\Q$ by \cite[Proposition 11, \S 4.6]{Car}. 
By construction, the image of $D$ under the composition
$$
\Div(Y)_\Q\to N^1(Y_{K})_\Q\to N^1(Y_{\overline{K}})_\Q
$$
coincides with the image of $\b$. But $N^1(Y_{K})_\Q\to N^1(Y_{\overline{K}})_\Q$ is injective by the projection formula, and the result follows. 
\end{proof}

\begin{lem}\label{lem:gg} Let $\pi:\fY\to B$ be a projective and flat morphism, with $B$ a smooth projective curve over $k$. If $\fL\in\Pic(\fY)$ is ample on the generic fiber of $\pi$, then there exists an ample line bundle $H$ on $B$ such that $\cO_B(mH)\otimes \pi_*\cO_\fX(m\fL)$ is globally generated for all $m$ sufficiently large and divisible. 
\end{lem}
\begin{proof} Set $\cF_m:=\pi_*\cO_\cX(m\fL)$, and pick a very ample line bundle $H$ on $B$. By the Castelnuovo-Mumford criterion~\cite[Theorem 1.8.5]{Laz} it is enough to show the existence of $m_0\in\N$ such that 
\begin{equation}\label{equ:h1}
H^1\left(B,\cO_B(mm_0H)\otimes\cF_m\right)=0
\end{equation}
for all $m$ large and divisible.  

Since $\fL$ is ample over the generic point of $B$, the $\cO_B$-algebra $\bigoplus_{m\ge 0}\cF_m$ is finitely generated at the generic point of $B$. After perhaps replacing $\fL$ by $d\fL$ for some $d\in\N$, we may further assume that the generators have degree $1$, so that $\cF_m/\cF_1^m$ has zero-dimensional support for all $m\ge 1$. As a consequence, the map
$$
H^1\left(B,\cO_B(mm_0H)\otimes\cF_1^m\right)\to H^1\left(B,\cO_B(mm_0H)\otimes\cF_m\right)
$$
is surjective for all $m\ge 1$. Upon replacing $\fX$ with $\Proj_B\left(\bigoplus_{m\ge 0}\cF_1^ m\right)$ 
we are thus reduced to proving (\ref{equ:h1}) when $\fL$ is $\pi$-ample, \ie ample on \emph{all} fibers of $\pi$. In that case we have $R^q\pi_*\cO_\fX(m\fL)=0$ for $m\gg 1$ and $q>0$ by Serre vanishing, and the degeneration of the Leray spectral sequence yields
$$
H^1\left(B,\cO_B(mm_0H)\otimes\cF_m\right)\simeq H^1\left(\fX,\cO_\fX(m(\fL+m_0\pi^*H)\right),
$$
which vanishes for all $m\gg 1$ if we choose $m_0$ such that $\fL+m_0\pi^*H$ is ample on $\fY$. 
\end{proof}

%
%
%%%%%%%%%%%%%%%%%%%%%%%%%%%%%%%%%%%%%%%%%%%%%%%%%%%%%%%%%%%%%%%%%%%
%
%

% 
%
%%%%%%%%%%%%%%%%%%%%%%%%%%%%%%%%%%%%%%%%%%%%%%%%%%%%%%%%%%%%%%%%%%%
%
%

\end{document}